\documentclass[11pt]{article}

\usepackage[pagebackref,colorlinks=true,pdfpagemode=none,urlcolor=blue,linkcolor=blue,citecolor=blue]{hyperref}

\usepackage{amsthm, amssymb, bm, bbm}
\usepackage{soul, color,cancel}
\usepackage[]{amsmath}
\usepackage[]{amsfonts}
\usepackage[]{fancyhdr}
\usepackage[]{graphicx}
\graphicspath{{/EPSF/}{./}{../figures/}{./figures/}}
\usepackage[colorinlistoftodos,disable]{todonotes}

\newcommand{\hd}{H^\text{sol}}

\newtheorem{theorem}{Theorem}[]
\newtheorem{lemma}{Lemma}[]
\newtheorem{proposition}[theorem]{Proposition}

\newtheorem{remark}{Remark}[]

\def \Cm {\mathbb{C}}
\def \Dm {M}

\def \Rm {\mathbb{R}}
\def \Sm {\mathbb{S}}

\def \Zm {\mathbb{Z}}

\newcommand      {\abar}        {{ \overline a}}

\def\B{\mathcal{B}}
\def\C{\mathcal{C}}

\def\G{\mathcal{G}}

\def\I{\mathcal{I}}

\def\O{\mathcal{O}}

\def\SS{\mathcal{S}}

\def\V{\mathcal{V}}

\def\Bh{\overset{\rightharpoonup}{\vphantom{a}\smash{\mathcal{B}}}}
\def\Ba{\overset{\leftharpoonup}{\vphantom{a}\smash{\mathcal{B}}}}
\def\Dh{\overset{\rightharpoonup}{\vphantom{a}\smash{\mathcal{D}}}}
\def\Da{\overset{\leftharpoonup}{\vphantom{a}\smash{\mathcal{D}}}}
\def\uh{\overset{\rightharpoonup}{\vphantom{a}\smash{u}}}
\def\ua{\overset{\leftharpoonup}{\vphantom{a}\smash{u}}}
\def\vh{\overset{\rightharpoonup}{\vphantom{a}\smash{v}}}

\newcommand{\cout}[1]{}

\newcommand{\x}{\mathrm{x}}

\newcommand{\sgn}[1]{\,{\rm sign}(#1)}

\newcommand{\dprod}[2]{\left\langle{#1},{#2}\right\rangle}
\newcommand{\dbar}{\overline{\partial}}
\newcommand{\zbar}{\overline{z}}

\title{Efficient tensor tomography in fan-beam coordinates. \\ II: Attenuated transforms}
\author{Fran\c{c}ois Monard \thanks{Department of Mathematics, University of California Santa Cruz, 1156 High Street, Santa Cruz, CA 95064. Email: fmonard@ucsc.edu} \thanks{Partial support from NSF grant DMS-1712790 is gratefully acknowledged. The author thanks the anonymous referees for valuable comments.}}

\begin{document}
\maketitle

\begin{abstract}
    This article extends the author's past work \cite{Monard2015a} to attenuated X-ray transforms, where the attenuation is complex-valued and only depends on position. We give a positive and constructive answer to the attenuated tensor tomography problem on the Euclidean unit disc in fan-beam coordinates. For a tensor of arbitrary order, we propose an equivalent tensor of the same order which can be uniquely and stably reconstructed from its attenuated transform, as well as an explicit and efficient procedure to do so.  
\end{abstract}

\section{Introduction}

We present a sequel to \cite{Monard2015a}, concerned with the reconstruction of tensor fields from their X-ray transform, to the case of transforms with attenuation. Let $M = \{ \x=(x,y) \in \Rm^2,\ x^2+y^2\le 1 \}$ the Euclidean unit disc, $a\in C^0(M,\Cm)$ and $SM = M\times \Sm^1$ the unit circle bundle of $M$. The variable in $\Sm^1$ will be referred to as ``angular''. For $f\in L^2(SM)$, we define the {\em attenuated X-ray transform} 
of $f$ by
\begin{align}
  I_a f(\x,v) = \int_0^{\tau(\x,v)} f(\varphi_t(\x,v)) \exp\left( \int_0^{t} a(\gamma_{\x,v}(s))\ ds \right)\ dt, \quad (\x,v) \in \partial_+ SM,
  \label{eq:attraytrans}
\end{align}    
where we denote $\varphi_t(\x,v) = (\x+tv, v) = (\gamma_{\x,v}(t), v)$ the Euclidean geodesic flow and 
\begin{align*}
    \partial_+ SM = \{ (\x,v)\in SM, \quad \x\in \partial M, \quad v\cdot \nu_\x \ge 0 \}
\end{align*}    
denotes the ingoing boundary (data space where the transform is defined). Considering an integer $m$, we denote by $L^2_{(m)}(SM)$ the subspace of $L^2(SM)$ consisting of elements with harmonic content in the angular variable contained in $\{-m,\dots,m\}$.

Upon restricting it to certain subspaces $L^2_{(m)}(SM)$, the transform \eqref{eq:attraytrans} encompasses several problems: when $f(\x,v) = f(\x)$ ($m=0$), this is the mathematical formulation of SPECT \cite{Arbuzov1998,BomanStroemberg2004,Finch2004,Natterer2001a,Novikov2002}; when $f$ is a vector field (i.e., linear in $v$ in the above form, or $m=1$), this is the mathematical formulation of Doppler Tomography \cite{Sparr1995,Holman2009,Kazantsev2007,Sadiq2014}; for general $m$, this corresponds to symmetric $m$-tensors and the tensor tomography problem \cite{Sharafudtinov1994,Sadiq2015,Paternain2011a}; finally, for integrands with non-polynomial dependence (or infinite harmonic content in the angular variable), this problem has applications in the study of the Boltzmann transport equation, with applications to Bioluminescence Imaging \cite{Bal2007,Stefanov2008a}, as will be illustrated in forthcoming work. While the first two cases have now been studied for a few decades, the present setting allows for a comprehensive understanding of the problem for integrands with general angular dependence; moreover, the fan-beam viewpoint adopted (traditionally rebinned into parallel geometry in the Euclidean case, see \cite{Natterer2001}) allows to elucidate certain questions related to X-ray transforms on manifolds, which continue to receive serious attention \cite{Assylbekov2017,Monard2015,Paternain2012,Paternain2011a,Paternain2015}.  


The attenuated tensor tomography problem we consider may be formulated as follows: given $m$ and $f\in L^2_{(m)}(SM)$, what is reconstructible of $f$ from $I_a f$, and how to reconstruct it ?  


Except for isolated cases, the operator $I_a$ restricted to $L^2_{(m)}(SM)$ has a non-trivial kernel given by the following: for any $h\in L^2_{(m-1)}(SM)$ with spatial components in $H^1_0(M)$, $I_a [(X+a)h] = 0$ (with $X = \cos\theta \partial_x + \sin\theta \partial_y$ the geodesic vector field). It is natural to ask whether these integrands are the only elements in the kernel, and we provide a positive answer to this question. Further, the size of the kernel of $I_a$ restricted to such tensors increases, and for reconstruction purposes, we present a candidate to be reconstructed modulo this kernel. As can be expected heuristically from \cite{Monard2015a}, when considering tensors of order $m\ge 1$, the form of the reconstructed candidate only differs from the case $m=1$ by residual terms. By ``residual'' here we mean that such terms are harmonic in position, and therefore represent very little relevant information compared to the main ``bulk'' made up of two full functions in $L^2(M)$ and $H^1_0(M)$. In particular, these residual terms contain no singularity inside the domain. The reconstruction procedure then consists in reconstructing the residual terms first, then the bulk. 

The reconstruction of the residual terms requires the explicit construction of {\em invariant distributions}, that is, distributional solutions of $Xw = 0$ on $SM$, with conditions on their moments (e.g., fiberwise holomorphic with prescribed fiberwise average). The quest for such invariant distributions has been rather active \cite{Paternain2012a,Paternain2015}, as injectivity statements of X-ray transforms have been proven to be equivalent to the existence of certain invariant distributions \cite{Paternain2016}, which can also be formulated as surjectivity results for backprojection operators (e.g., $I_0^*$). A salient feature here is the explicit construction of such distributions, whose existence is usually based on ellipticity arguments \cite{Pestov2005}, or series of iterated Beurling transforms in higher dimensions \cite{Paternain2015}. Such invariant distributions, via appropriate integrations by parts on $SM$, allow to obtain reconstruction formulas for the residual terms mentioned above.  

We then carry out the reconstruction of the main bulk ($g_0$ and $g_s$ below). In a recent work with Assylbekov and Uhlmann \cite{Assylbekov2017}, the author provided range characterizations and reconstruction formulas for the attenuated ray transform on surfaces, restricted to the case $m=1$ (i.e., sums ``function + vector field'') in smooth topologies, in particular generalizing the approach in \cite{Kazantsev2007} to complex-valued attenuations and non-Euclidean geometries following \cite{Salo2011,Monard2015}. We revisit these results here, and the Euclidean case allows for more precise statements. The main tools involved are a {\em holomorphization operator} (as introduced in \cite{Assylbekov2017}) and {\em holomorphic integrating factors} first introduced in \cite{Salo2011} which are so crucial in two-dimensional tomography problems \cite{Assylbekov2017,Paternain2011a,Paternain2012}. 

It should be noted that this reconstructed representative is also valid for vector fields alone (Doppler transform), though unlike the reconstruction formulas provided in \cite{Kazantsev2007,Monard2015}, this formula provides a partial reconstruction of a vector field (namely, its solenoidal part) even where the attenuation vanishes. This is similar in spirit to \cite{Tamasan2007}, see also Remark \ref{rem:solenoidal}.


We restrict this article to the case of attenuations which depends on position only. Other types are considered in other settings, that is, linearly-dependent in angle \cite{Paternain2012,Monard2016}. Such ``attenuations'' have a different physical meaning (that of a connection, see \cite{Paternain2012b,Paternain2012,Zhou2017}), and their inversion can sometimes be tackled without the use of holomorphic integrating factors, see the recent work \cite{Monard2016}. 


Other approaches for tackling Euclidean attenuated transforms have been $A$-analytic function theory {\it \`a la} Bukhgeim \cite{Arbuzov1998,Kazantsev2007,Tamasan2007}, leading in particular to recent range characterizations of the attenuated transform over functions, one-forms and second-order tensors \cite{Sadiq2015,Sadiq2013,Sadiq2014}, and Riemann-Hilbert problems, leading in particular to an efficient reconstruction of functions from their attenuated transform \cite{BomanStroemberg2004,Natterer2001a,Novikov2002} and studies of more general integrands and partial data problems in \cite{Bal2004a}.



We finally point out a few differences with the previous work \cite{Monard2015a} on unattenuated tensor tomography:
\begin{itemize}
  \item The decompositions modulo kernel presented no longer split according to the parity of the tensor order $m$, as attenuated transport equations now mix all even and odd angular modes. 
  \item In order to recover the residual elements, the approach in \cite{Monard2015a} was to compute their forward transform, whose frequency content in data space could easily be described, and an inversion formula easily derived. In the attenuated case, this is no longer the case, hence the necessity of constructing special invariant distributions, combined with integrations by parts on $SM$. 
  \item Unlike the attenuated case, residual elements must be reconstructed in a specific order, that is, from higher to lower.
\end{itemize}
We now state the main resuts. 

\section{Main results}

\paragraph{Spaces and notation.} Denote by $C_P$ the {\em Poincar\'e constant}\footnote{Via Rayleigh quotient, $C_p = \lambda_1^{-1}$, where $\lambda_1>0$ is the smallest eigenvalue of the Dirichlet Laplacian on $M$, see \cite[Ch. 11]{Strauss2007}} of the unit disc, that is,
\begin{align}
    C_P := \max \{C >0 :\ \|u\|_{L^2(M)}^2 \le C \|\nabla u\|_{L^2(M)}^2, \quad \forall u\in H^1_0(M)\}.     
    \label{eq:CP}
\end{align}    
The spaces $L^2(SM)$ and $L^2(M)$ are endowed with their usual inner products denoted $\dprod{\cdot}{\cdot}_{SM}$ and $\dprod{\cdot}{\cdot}_{M}$ and corresponding norms 
\begin{align*}
    \|u\|^2 := \int_{SM} |u(\x,\theta)|^2\ d\x\ d\theta, \quad u\in L^2(SM), \qquad \|f\|_M^2 := \int_M |f(\x)|^2 \ d\x, \quad f\in L^2(M).
\end{align*}
We decompose $L^2(SM)$ into circular harmonics
\begin{align}
    L^2(SM) = \bigoplus_{k=0}^\infty H_k, \qquad H_k := \ker(\partial_\theta^2 + k^2 Id),
    \label{eq:harmonics}
\end{align}
where $H_0$ is isometric to $L^2(M)$ and for $k\ge 1$, $f\in H_k$ if and only if $f(x,\theta) = f_{k,+}(\x) e^{ik\theta} + f_{k,-}(\x) e^{-ik\theta}$ for some functions $f_{k,\pm}\in L^2(M)$. An element $u\in L^2(SM)$ decomposes accordingly
\begin{align*}
    u(\x,\theta) &= \sum_{k=0}^\infty u_k(\x,\theta) = u_0(\x) + \sum_{k=1}^\infty (u_{k,+}(\x)e^{ik\theta} + u_{k,-}(\x) e^{-ik\theta}), \\
    \|u\|^2 &= 2\pi \left( \|u_0\|_M^2 + \sum_{k=1}^\infty (\|u_{k,+}\|_M^2 + \|u_{k,-}\|_M^2) \right). \qquad (\text{Parseval})
\end{align*}
We also denote $L^2_{(m)}(SM) := \bigoplus_{k=0}^m H_k$, such a space can be viewed as the restriction to $SM$ of sums of symmetric tensor fields of order up to $m$. Denote $L^2(\ker\overline{\partial})$ and $L^2(\ker\partial)$ the subspaces of $L^2(M)$ made of complex- analytic and antianalytic functions (where $\partial = \frac{1}{2} (\partial_x - i\partial_y)$ and $\dbar = \frac{1}{2} (\partial_x + i\partial_y)$). Such spaces are closed in $L^2(M)$ (see e.g. \cite[Ex. 6 p254]{SteinShakarchi2005}), and for $k\ge 1$, we then define
\begin{align}
    \begin{split}
    \hd_k &:= \{ f = e^{ik\theta} f_{k,+}(\x) + e^{-ik\theta} f_{k,-}(\x),\ f_{k,-} \in L^2(\ker\partial), \ f_{k,+} \in L^2(\ker \overline{\partial}) \} \\
    &= \hd_{k,+} \oplus \hd_{k,-}.
    \end{split}
    \label{eq:Hdelta}
\end{align}
By the observations above, $\hd_k$ is a closed subspace of $H_k$ (hence Hilbert).

\begin{remark} The space $\hd_k$ corresponds to restrictions to $SM$ of trace-free (if $k\ge 2$), divergence-free symmetric $k$-tensors. To see this, decompose a $k$-tensor into $f = \sum_{p=0}^k f_p\ \sigma (dz^p\otimes dz^{k-p})$, whose trace is computed as 
    \begin{align*}
	\text{tr } f(v_1,\dots,v_{k-2}) = f(v_1,\dots,v_{k-2}, \partial_x, \partial_x) + f(v_1,\dots,v_{k-2}, \partial_y, \partial_y) &= 4 f(v_1,\dots,v_{k-2}, \partial, \dbar).
    \end{align*}
    Then $\text{tr }\sigma (dz^p\otimes dz^{k-p})$ is nonzero proportional to $\sigma(dz^{p-1} \otimes d\zbar^{k-1-p})$ if $0<p<k$, and zero otherwise. In particular, a trace-free $k$-tensor takes the form $f_0 dz^k + f_k d\zbar^k$, and its divergence is then given by $\text{div } f = (\dbar f_0) dz^{k-1} + (\partial f_k) d\zbar^{k-1}$, see, e.g., \cite[Appendix B]{Paternain2015}. Hence the claim. 
\end{remark}

\begin{remark} For the sake of brevity in Theorem \ref{thm:rep} and its proof, we are changing notation slightly from \cite{Monard2015a}, keeping circular harmonics of same magnitude $\pm k$ in the same subspace $H_k$. In this correspondence and recalling the definitions $\eta_+ = e^{i\theta} \partial$ and $\eta_- = e^{-i\theta} \overline{\partial}$, $\hd_{k,+}$ corresponds to $L^2(\ker^k \eta_-)$ in \cite{Monard2015a}, and $\hd_{k,-}$ corresponds to $L^2(\ker^{-k} \eta_+)$.
\end{remark}

The inward boundary $\partial_+ SM$ introduced in Sec. \ref{sec:SM} is parameterized with {\em fan-beam coordinates} $(\beta,\alpha) \in \Sm^1 \times (-\pi/2,\pi/2)$  (see Fig. \ref{fig:fanbeam}) together with the inner product 
\begin{align*}
    \left\langle f,g\right\rangle_{\partial_+ SM} = \int_{\partial_+ SM} f(\beta,\alpha) \overline{g(\beta,\alpha)} \ d\beta\ d\alpha.
\end{align*}

For $a\in C^0(M,\Cm)$ and $f\in C^0(SM)$, then $I_a f \in C^0(\partial_+ SM)$, moreover the following mapping property is immediate. 
\begin{lemma}\label{lem:continuity} For $a\in C^0(M,\Cm)$, the attenuated ray transform extends into a bounded operator $I_a:L^2(SM)\to L^2(\partial_+ SM)$, with norm no greater than $\sqrt{2} e^{2a_\infty}$, where $a_\infty := \sup_{\x\in M} |a(\x)|$.    
\end{lemma}

\begin{proof} This is a direct consequence of \cite[Lemma 4.1]{Monard2015a} and the obvious pointwise estimate $|I_a f(\beta,\alpha)| \le e^{2a_\infty} I |f| (\beta,\alpha)$, where the factor $2$ is the diameter of $M$ and $I$ denotes the unattenuated transform.     
\end{proof}

\paragraph{Gauge representatives of the attenuated transform. }

We denote $X = \cos\theta \partial_x + \sin\theta \partial_y$ the geodesic vector field and $X_\perp = \sin\theta \partial_x - \cos\theta \partial_y$ the ``transverse derivative''. 

\begin{theorem}\label{thm:rep}
    Let $a\in C^0(M,\Cm)$ with supremum $a_\infty$ and let $m$ a natural integer. For any $f\in L^2_{(m)}(SM)$, there exists $g\in L^2_{(m)}(SM)$, linear in $f$, satisfying $I_a f = I_a g$, of the form 
  \begin{align}
      g = g_0 + X_\perp g_s + \sum_{k=1}^{m} g_k,\quad (g_0,g_s) \in L^2(M)\times H_0^1(M),\quad g_k\in \hd_k, \quad 1\le k\le m.
    \label{eq:g}
  \end{align}	
  Moreover, with $C:= 4+8a_\infty^2 C_P$, we have the following stability estimates
  \begin{align}
      \begin{split}
	  \|g\|^2 &\le 2 \|f_0\|^2 + (1+4a_\infty^2 C_P)\|f_1\|^2, \qquad (m=1), \\
	  \|g\|^2 &\le 4 \sum_{p=0}^{m-2} C^p \|f_p\|^2 + 2 C^{m-1} \|f_{m-1}\|^2 + C^m \|f_m\|^2, \qquad (m\ge 2).        \end{split}
      \label{eq:repest}
  \end{align}      
\end{theorem}

Theorem \ref{thm:rep} was established in the unattenuated case \cite{Monard2015a}, though the dependence of the continuity estimate on $m$ is now made explicit. Note that such an estimate for general $f\in L^2(SM)$ could not be possible, unless we assume some decay on the angular moments of $f$. To this end, let us define, for $\kappa \ge 1$
\begin{align*}
    L^{2,\kappa}(SM) &:= \{f\in L^2(SM),\  \|f\|_{\kappa} <\infty \}, \\
    \text{where}\qquad \|f\|^2_{\kappa} &:= 2\pi (\|f_0\|_M^2 + \sum_{k\ge 1} \kappa^k (\|f_{k,+}\|_M^2 + \|f_{k,-}\|_M^2)).
\end{align*}

\begin{theorem}\label{thm:repinf} Suppose $f\in L^{2,C} (SM)$ where $C$ is the constant in Theorem \ref{thm:rep}. Then there exists $g\in L^2(SM)$, linear in $f$, satisfying $I_a f = I_a g$, of the form 
    \begin{align}
	g = g_0 + X_\perp g_s + \sum_{k=1}^{\infty} g_k,\quad g_0 \in L^2(M), \quad g_s \in H_0^1(M),\ \quad g_k\in \hd_k, \quad k\ge 1.
	\label{eq:ginf}
    \end{align}
    Moreover, we have the estimate
    \begin{align}
	\|g\|^2 \le 4 \|f\|^2_{C}.
	\label{eq:repestinf}
    \end{align}
\end{theorem}

\paragraph{Uniqueness and reconstruction.}
We now explain the reconstruction procedure, which requires introducing additional kernels and concepts. We define the kernel constructed in \eqref{eq:Green} 
\begin{align*}
    \G (x, y; \beta,\alpha) := \frac{1}{4\pi^2} \left( \frac{e^{-i\alpha}}{(1+ze^{-i(\beta+2\alpha)})^2} + \frac{e^{i\alpha}}{(1-z e^{-i\beta})^2} \right). \qquad z = x+iy.
\end{align*}
Also introduce the so-called {\em holomorphization operator} (see \eqref{eq:B}) for $h\in L^2(SM)$
\begin{align*}
    \Bh h := \frac{1}{2} \left[ (Id - iH) h + i (Id + iH) (A_+ P^\dagger A_-^\star (Id-iH)h) \right]|_{\partial_+ SM}, \qquad\text{and}\qquad \Ba h:= \overline {\Bh \overline{h}},
\end{align*}
with $A_\pm, A_\pm^*$ defined in Sec. \ref{sec:SM} and $P^\dagger$ in Lemma \ref{lem:Pdagger}. We finally define a {\em holomorphic integrating factor} for $a$ (see Prop. \ref{prop:HIF}) by 
\begin{align*}
    w_a := \frac{-i}{4} (Id + iH) (P_-^* I_0 a)_\psi, \qquad \rho_a := w_a|_{\partial_+ SM} = \frac{1}{2}I_0 a - \frac{i}{4} P_-^* I_0 a,
\end{align*}
with $P_-^*$ defined in \eqref{eq:Pstardef}. The reconstruction procedure then goes as follows.

\begin{theorem}\label{thm:reconstruction}
    Considering $(a,f,g,m)$ as in Theorem \ref{thm:rep}, the representative $g$ is unique and is reconstructed as follows. For $k=m$ down to $1$, reconstruct $g_k = e^{ik\theta} g_{k,+} + e^{-ik\theta} g_{k,-} \in \hd_k$ via
    \begin{align*}
	g_{k,+} (\x) &= (-1)^k \int_{\partial_+ SM} e^{-\overline{\rho_{\abar}}}\ \I_k (\beta,\alpha)\  e^{-ik (\beta+\alpha)}\ \G (x,y; \beta,\alpha) \ d\beta\ d\alpha. \\
	g_{k,-}(\x) &= (-1)^k \int_{\partial_+ SM} e^{-\rho_{a}}\ \I_k (\beta,\alpha)\  e^{ik (\beta+\alpha)}\ \overline{\G (x,y; \beta,\alpha)} \ d\beta\ d\alpha.
    \end{align*}
    where $\I_k := I_a (g-\sum_{\ell = k+1}^m g_\ell)$. Then defining $\I := I_a (g-\sum_{\ell = 1}^m g_\ell) = I_a(g_0 + X_\perp g_s)$, the two functions $g_0,g_s$ are reconstructed via
    \begin{align*}
	g_0 &= -\eta_+ \Dh_{-1} - \eta_- \Da_1 - \frac{a}{2} \left( \Dh_0 + \Da_0 + i(g_+ - g_-) \right), \\
	g_s &= \frac{1}{2} (g_+ + g_-) - \frac{i}{2} (\Dh_0 - \Da_0),
    \end{align*}
    where we have defined $\Dh := e^{w_a} (\Bh(\I e^{-\rho_a}))_\psi$, $\Da := e^{\overline{w_{\overline{a}}}} (\Ba(\I e^{-\overline{\rho_{\overline{a}}}}))_\psi$, and where $\partial g_+ = \dbar g_- = 0$, so that $g_\pm$ are expressed as Cauchy formulas in terms of their known boundary conditions
    \begin{align*}
	g_+|_{\partial M} = -i (\I - \Dh|_{\partial SM})_0, \qquad g_-|_{\partial M} = i (\I - \Da|_{\partial SM})_0.
    \end{align*}
\end{theorem}

\begin{remark}[Connection with the Doppler transform] \label{rem:solenoidal} We briefly explain how Theorem \ref{thm:reconstruction} applies to the Doppler transform, i.e., the attenuated transform over a vector field: a vector field $V\in L^2$, restricted to $SM$, takes the form $V = v_1(\x) e^{i\theta} + v_{-1}(\x) e^{i\theta}$. Applying a Helmholtz decomposition to $V$, we may write $V = Xf + X_\perp g$ for $f\in H^1_0(M)$ and $g\in \dot{H}^1(M)$. Then a direct calculation shows that $I_a V = I_a (-af + X_\perp g)$, out of which Theorem \ref{thm:reconstruction} implies that $-af$ and $g$ can both be reconstructed uniquely and stably in all cases. This allows to recover the following facts:
    \begin{itemize}
	\item As previously established in \cite{Kazantsev2004}, if $a$ vanishes nowhere, then both $f$ and $g$, and thus $V$, can be reconstructed stably throughout the domain. In general, this also clarifies why $f$ (or the potential part of $V$) can only be recovered on the support of $a$.
	\item As previously established in \cite{Tamasan2007}, $\text{curl }V = \Delta g$ can be reconstructed everywhere regardless of whether the attenuation vanishes. 
    \end{itemize}    
\end{remark}

Theorems \ref{thm:rep} and \ref{thm:reconstruction} motivate the following result. Such a result was proved before in the case of smooth tensors for the magnetic ray transform in \cite{Ainsworth2013}.

\begin{theorem}\label{thm:injectivity}
    For $f\in L^2_{(m)}(SM)$, if $I_a f = 0$, then there exists $h\in L^2_{(m-1)}(SM)$ with components in $H^1_0(M)$, vanishing at $\partial SM$, such that $f = Xh + ah$. 
\end{theorem}

{\bf Outline and additional results.} The rest of the article is organized as follows. In Section \ref{sec:prelim}, we recall preliminaries on the geometry of $SM$ and its boundary (\S \ref{sec:SM}), and on transport equations (\S \ref{sec:transport}). We prove Theorems \ref{thm:rep}, \ref{thm:repinf} and \ref{thm:injectivity} in Section \ref{sec:proofs124}. The reconstruction aspects are covered in Section \ref{sec:reconstruction}. We first describe important boundary operators in \S \ref{sec:boundary}, namely $P_\pm$ (defined in \eqref{eq:P}) and provide explicit pseudo inverses $P_\pm^\dagger$ for them in Lemma \ref{lem:Pdagger}. We then show in Proposition \ref{prop:FBP} that $P_\pm^\dagger$ are the ``filters'' in the (unattenuated) filtered-backprojection formulas for functions and solenoidal vector fields in fan-beam coordinates, allowing as a novelty for vector fields to be supported up to the boundary. This then allows us to construct holomorphic integrating factors (\S \ref{sec:HIF}) for functions and vector fields in Propositions \ref{prop:special} and \ref{prop:HIF}. Next in Section \ref{sec:invariant}, we cover in Theorem \ref{thm:I0star} the explicit construction of invariant distributions with one-sided harmonic content and prescribed fiberwise average. Using such distributions via integration by parts on $SM$, we then show in \S \ref{sec:gm} how to reconstruct the residual terms $g_m \in \hd_m$. Finally, we cover the reconstruction of the last two functions $(g_0,g_s)$ in \S \ref{sec:g0gs} in Theorem \ref{thm:inversionfh0}, defining and using along the way the holomorphization operator in Proposition \ref{prop:holomorphization}.


\section{Preliminaries} \label{sec:prelim}

\subsection{Geometry of $SM$ and its boundary} \label{sec:SM}

The generator of all lines in the unit disc is the geodesic vector field $X = \cos\theta \partial_x + \sin\theta \partial_y$. $X$ is completed into a global frame of $SM$ using $X_\perp := \sin\theta \partial_x - \cos\theta \partial_y$ and $V:= \partial_\theta$. In the harmonic decomposition \eqref{eq:harmonics}, we define the {\em fiberwise Hilbert transform} $H:L^2(SM)\to L^2(SM)$ by the formula
\begin{align*}
    H (f(\x) e^{ik\theta}) = -i\sgn{k} f(\x) e^{ik\theta}, \quad f\in L^2(M), \quad k\in \Zm \quad (\text{with the convention } \sgn{0} = 0).
\end{align*}    
We also denote $H:L^2(\partial SM)\to L^2(\partial SM)$ the operator defined by the same formula on the circles sitting above points at the boundary $\partial M$. The following commutator was derived in \cite{Pestov2005}:
\begin{align}
    [H,X] u = X_\perp u_0 + (X_\perp u)_0, \qquad u\in C^\infty(SM), 
    \label{eq:commutator}
\end{align}    
where we define {\em fiberwise average} $u_0(\x) := \frac{1}{2\pi} \int_0^{2\pi} u(\x,\theta)\ d\theta$. $H$ splits into $H= H_+ + H_-$ where $H_{+/-}$ denotes $H$ restricted to even/odd harmonics. A function $u\in L^2(SM)$ is called {\em fiberwise holomorphic} (resp {\em antiholomorphic}) if $(Id+iH)u = 0$ (resp. $(Id-iH)u =0$). Fiberwise holomorphic functions are preserved under product and exponentiation (when they are defined in appropriate topologies).  

\todo[inline]{Does $H$ maps $C^0$ into $C^0$ ?}

We parameterize the boundary $\partial SM$ with {\em fan-beam coordinates} $(\beta,\alpha)\in \Sm^1 \times \Sm^1$, where $\x(\beta) = \binom{\cos\beta}{\sin\beta}$ parameterizes the point at the boundary and $v = \binom{\cos(\beta+\pi+\alpha)}{\sin (\beta+\pi+ \alpha)}$ is a vector at the basepoint $\x(\beta)$. The inward boundary $\partial_+ SM$ is the subset of $\partial SM$ for which $\frac{-\pi}{2}\le \alpha \le \frac{\pi}{2}$ (making $v$ inward-pointing) and the outward boundary $\partial_- SM$ is the subset of $\partial SM$ for which $\frac{\pi}{2}\le \alpha \le \frac{3\pi}{2}$ (making $v$ outward-pointing). The {\em scattering relation} $\SS:\partial_\pm SM \to \partial_\mp SM$ maps an ingoing/outgoing point to the outgoing/ingoing other end of the unique geodesic passing through it, and the {\em antipodal scattering relation} $\SS_A:\partial_+ SM\to \partial_+ SM$ maps an ingoing point to that at the other of the unique geodesic passing through it. These maps are given explicitly by 
\begin{align}
    \SS(\beta,\alpha) = (\beta+\pi + 2\alpha, \pi-\alpha), \qquad \SS_A (\beta,\alpha) = (\beta+ \pi + 2\alpha, -\alpha).
    \label{eq:scatrel}
\end{align}    

\begin{figure}[htpb]
    \centering
    \includegraphics[height=0.26\textheight]{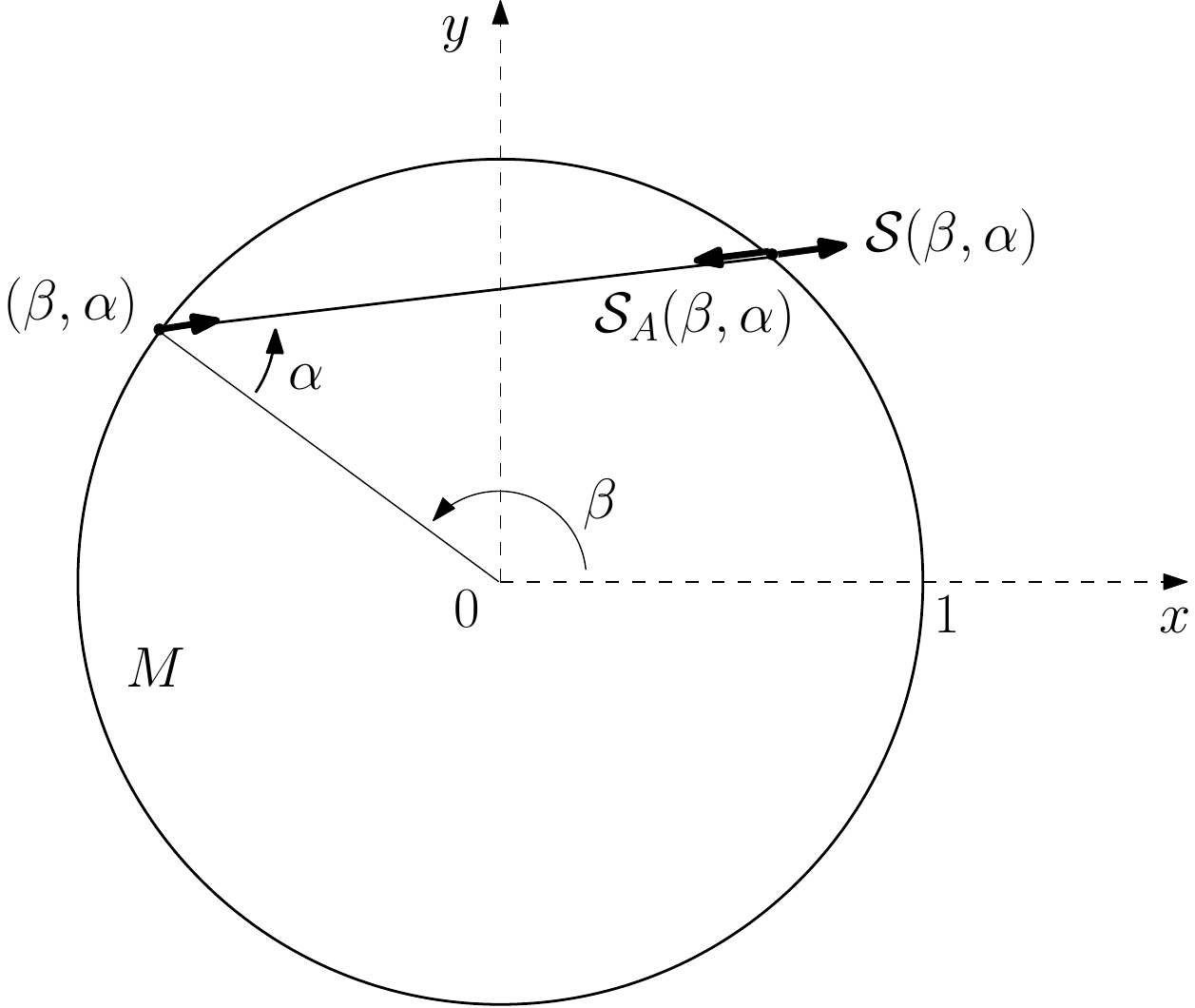}
    \caption{Fan-beam coordinates, scattering relation $\SS$ and antipodal scattering relation $\SS_A$.}
    \label{fig:fanbeam}
\end{figure}

We define $A_{\pm} : L^2(\partial_+ SM)\to L^2(\partial SM)$ the operators of even/odd extension via scattering relation, i.e. 
\begin{align*}
    A_\pm u(\beta,\alpha) = \left\{
    \begin{array}{ll}
	u(\beta,\alpha), & (\beta,\alpha)\in \partial_+SM, \\
	\pm u( \SS(\beta,\alpha)), & (\beta,\alpha)\in \partial_-SM,
    \end{array}
    \right. \qquad u\in L^2(\partial_+ SM). 
\end{align*}
Their adjoints are given by 
\begin{align*}
    A_\pm^* u (\beta,\alpha) = u(\beta,\alpha) \pm u(\SS(\beta,\alpha)), \qquad (\beta,\alpha)\in \partial_+ SM.
\end{align*}

\subsection{Transport equations and integration by parts on $SM$} \label{sec:transport}

The transform \eqref{eq:attraytrans} may be realized as the influx trace $u|_{\partial_+ SM}$ of the unique solution $u = u_a^f$ to the transport problem
\begin{align*}
  Xu + au = -f\quad (SM), \qquad u|_{\partial_- SM} =0. \qquad \qquad (\text{here, } X = \cos\theta \partial_x + \sin\theta \partial_y).
\end{align*}
Recall Santal\'o's formula
\begin{align}
    \int_{SM} f(\x,\theta)\ d\x\ d\theta = \int_{\partial_+ SM} \cos\alpha \int_0^{2\cos\alpha} f(\varphi_t(\beta,\alpha))\ dt\ d\beta\ d\alpha.
    \label{eq:santalo}
\end{align}
For $h\in C^0 (\partial_+ SM)$, we denote $h_\psi = v \in C^0(SM)$ the unique solution to 
\begin{align}
    Xv = 0 \quad (SM), \qquad v|_{\partial_+ SM} = h.    
    \label{eq:hpsi}
\end{align}
Using \eqref{eq:santalo}, we can derive the following integration by parts formula, true for any $u,v\in C^1(SM)$: 
\begin{align*}
    \dprod{Xu}{v}_{SM} + \dprod{u}{X v}_{SM} = \int_{SM} X(u\overline{v}) &= \int_{\partial_+ SM} \int_0^{2\cos\alpha} \frac{d}{dt} (u\overline{v})(\varphi_t(\beta,\alpha))\ dt \cos\alpha\ d\alpha\ d\beta \\
    &= - \int_{\partial_+ SM} A_-^* \left( u\overline{v}|_{\partial SM} \right)\ \cos\alpha\ d\alpha\ d\beta. 
\end{align*}
In particular, if $u$ solves a transport equation of the form
\begin{align*}
    Xu + au = -f \qquad (SM), \qquad u|_{\partial_- SM} = 0, \qquad (\text{so that } u|_{\partial_+ SM} = I_a f),
\end{align*}
and if $w$ is a solution of $X w = -a$ with $\rho := w|_{\partial_+ SM}$, then the function $u' = u e^{-w}$ satisfies
\begin{align*}
    X u' = - f e^{-w}\quad (SM), \quad u'|_{\partial_+ SM} = e^{-\rho} I_a f, \qquad u|_{\partial_- SM} = 0. 
\end{align*}
In addition, if $v = h_\psi$ is as in \eqref{eq:hpsi}, then the integration by parts above applied to $u'$ and $v$ yields: 
\begin{align}
    \dprod{e^{-w} f}{v}_{SM} = \dprod{e^{-\rho} I_a f}{h\ \cos\alpha}_{\partial_+ SM}.
    \label{eq:IBPatt} 
\end{align}


\section{Gauges and uniqueness. Proofs of Theorems \ref{thm:rep}, \ref{thm:repinf} and \ref{thm:injectivity}.} \label{sec:proofs124}

{\bf Preliminaries.} Recall the notation $\eta_+ = e^{i\theta} \partial$ and $\eta_- = e^{-i\theta}\overline{\partial}$, and note that $X = \eta_+ + \eta_-$ while $X_\perp = \frac{1}{i} (\eta_+ - \eta_-)$. By virtue of the fact that $[\eta_+, \eta_-] = 0$ and $\eta_+^* = -\eta_-$, we have the following property: 
\begin{align}
    \|\eta_- u\|^2 = \|\eta_+ u\|^2, \qquad \forall u\in \C_0^\infty(SM).
    \label{eq:property}
\end{align}
By density, we can extend this property to integrands of the form $u = e^{ik\theta} v(\x)$, with $v\in H^1_0(M)$. In addition, for any such integrand, we have
\begin{align*}
    \|\eta_+ (e^{ik\theta} v(\x))\|^2 = 2\pi \|\partial v\|^2_{M} &= \frac{\pi}{2} \int_M (v_x - iv_y) (\bar{v}_x + i \bar{v}_y)\ d\x \\
    &= \frac{\pi}{2} \left( \|\nabla v\|_{M}^2 + i\cancel{ \int_M (v_x \bar{v}_y - v_y \bar{v}_x)\ d\x} \right),
\end{align*}
where the last term vanishes via Green's formula. In particular, using the Poincar\'e constant $C_P$ defined in \eqref{eq:CP}, we have for any $u = e^{ik\theta} v(\x)$ with $v\in H^1_0(M)$ and $k$ integer, 
\begin{align*}
    \|u\|^2 = 2\pi \|v\|^2_{M} \le 2\pi C_P \|\nabla v\|^2_{M} = 8\pi C_P \|\partial v\|_M^2 = 4C_p \| \eta_+ u \|^2, 
\end{align*}
and similarly $\|u\|^2_{L^2(SM)} \le 4C_p \| \eta_- u \|^2_{L^2(SM)}$. These estimates will be useful below to control the growth of constants.  

\begin{lemma}\label{lem:decomp} Let $f_k\in H_{k}$ for $k\ge 2$. Then there exists $g_k\in \hd_k$, $v_{k-1}\in H_{k-1}$ with $v_{k-1,\pm}\in H^1_0(M)$, and $w_{k-2}\in H_{k-2}$, such that
    \begin{align*}
	f_k = X v_{k-1} + w_{k-2} + g_k. 
    \end{align*}
Moreover, we have the continuity estimate $\|v_{k-1}\|^2 \le 4C_P \|f_k\|^2$ and 
    \begin{align*}
	\|w_{k-2}\|^2 + \|g_k\|^2 = \|f_k\|^2 \quad (k>2), \qquad \|w_{k-2}\|^2 + \|g_k\|^2 &\le 2\|f_k\|^2 \quad (k=2).
    \end{align*}    
\end{lemma}

\begin{proof} Let $f_k\in H_k$ and write $f_k = f_{k,+} e^{ik\theta} + f_{k,-} e^{-ik\theta}$. Using the elliptic decompositions associated with $\partial, \overline{\partial}$ operators, write
    \begin{align*}
	f_{k,+} &= \partial v_{k-1,+} + g_{k,+}, \qquad g_{k,+} \in L^2(M),\qquad \overline{\partial} g_{k,+} = 0, \qquad v_{k-1,+} \in H^1_0(M), \\
	f_{k,-} &= \overline{\partial} v_{k-1,-} + g_{k,-}, \qquad g_{k,-}\in L^2(M), \qquad \partial g_{k,-} = 0, \qquad v_{k-1,-}\in H^1_0(M),
    \end{align*}
    together with estimates
    \[ \|f_{k,+}\|_M^2 = \|\partial v_{k-1,+}\|_M^2 + \|g_{k,+}\|_M^2, \qquad \|f_{k,-}\|_M^2 = \|\overline{\partial} v_{k-1,-}\|_M^2 + \|g_{k,-}\|_M^2. \]
    Setting $g_k = g_{k,+}e^{ik\theta} + g_{k,-}e^{-ik\theta}$ and $v_{k-1} = v_{k-1,+} e^{i(k-1)\theta} + v_{k-1,-} e^{-i(k-1)\theta}$, this exactly means 
    \begin{align*}
	f_k &= \eta_+ (v_{k-1,+}e^{i(k-1)\theta}) + \eta_- (v_{k-1,-}e^{-i(k-1)\theta}) + g_k \\
	&= X v_{k-1} + g_k - \eta_- (v_{k-1,+}e^{i(k-1)\theta}) - \eta_+ (v_{k-1,-}e^{-i(k-1)\theta}). \qquad (X = \eta_+ + \eta_-)
    \end{align*}
    Setting $w_{k-2} := - \eta_- (v_{k-1,+}e^{i(k-1)\theta}) - \eta_+ (v_{k-1,-}e^{-i(k-1)\theta}) \in H_{k-2}$ with 
    \begin{align*}
	w_{k-2,-} = - \partial v_{k-1,-} , \qquad w_{k-2,+} = - \overline{\partial} v_{k-1,+}.
    \end{align*}
    If $k>2$, the sum is orthogonal and we have
    \begin{align*}
	\|w_{k-2}\|^2 + \|g_{k}\|^2 &= 2\pi \left( \|\partial v_{k-1,-}\|_M^2 + \|\overline{\partial} v_{k-1,+}\|_M + \|g_{k,+}\|_M^2 + \|g_{k,-}\|_M^2 \right) \\
	&= 2\pi \left( \|\overline{\partial} v_{k-1,-}\|_M^2 + \|\partial v_{k-1,+}\|_M + \|g_{k,+}\|_M^2 + \|g_{k,-}\|_M^2 \right) = \|f_k\|^2.
    \end{align*}
    The case $k=2$ is a direct consequence of the inequality $(a+b)^2 \le 2 (a^2 + b^2)$. On to the stability estimate for $v_{k-1}$, we have 
    \begin{align*}
	\|v_{k-1}\|^2 = 2\pi (\|v_{k-1,+}\|_M^2 + \|v_{k-1,-}\|_M^2)&\le 2\pi C_P (\|\nabla v_{k-1,+}\|_M^2 + \|\nabla v_{k-1,-}\|_M^2) \\
	&\le 4C_P\ 2\pi (\|\partial v_{k-1,+}\|_M^2 + \|\overline{\partial} v_{k-1,-}\|_M^2) \\
	&\le 4C_P \|f_k\|^2,
    \end{align*}
    hence the proof.
\end{proof}

We now prove Theorem \ref{thm:rep}.

\begin{proof}[Proof of Theorem \ref{thm:rep}] The case $m=0$ is trivial, with $g_0=f_0$ and continuity constant $C=1$. For the case $m=1$, writing $f = f_0 + f_1$, we decompose
    \begin{align*}
	f_{1,+} = \partial v_{0,+} + g_{1,+}, \qquad f_{1,-} = \overline{\partial} v_{0,-} + g_{1,-}, 
    \end{align*}
    where $v_{0,\pm}\in H^1_0(M)$ and $g_{1,\pm}\in L^2(M)$ with $\partial g_{1,-} = \overline{\partial} g_{1,+} = 0$, with stability estimates
    \begin{align}
	\|\partial v_{0,+}\|_M^2 + \|g_{1,+}\|_M^2 = \|f_{1,+}\|_M^2, \qquad \|\overline{\partial} v_{0,-}\|_M^2 + \|g_{1,-}\|_M^2 = \|f_{1,-}\|_M^2.
	\label{eq:norms}
    \end{align}
    With the identity
    \begin{align*}
	\eta_+ v_{0,+} + \eta_- v_{0,-} = X \left( \frac{v_{0,+} + v_{0,-}}{2} \right) + X_\perp \left( i\frac{v_{0,+}-v_{0,-}}{2} \right),
    \end{align*}
    we define $g_p := \frac{v_{0,+} + v_{0,-}}{2}\in H^1_0(M)$ and $g_s := i\frac{v_{0,+}-v_{0,-}}{2} \in H^1_0(M)$, so that we may rewrite $f$ as 
    \begin{align*}
	f = f_0 + X g_p + X_\perp g_s + g_1.
    \end{align*}
    Then the transport equation $Xu + au = -f$ can be rewritten as 
    \begin{align*}
	X (u + g_p) + a(u + g_p) = - (f_0 - a g_p) - X_\perp g_s - g_1 \quad (SM), \qquad (u+g_p)|_{\partial_- SM} = 0.
    \end{align*}
    Upon defining $g_0 := f_0 - a g_p$, and since $g_p$ vanishes on $\partial SM$, we have 
    \begin{align*}
	I_a (g_0 + X_\perp g_s + g_1) = (u+ g_p)|_{\partial_+ SM} = u|_{\partial_+ SM} = I_af. 
    \end{align*}
    Now for the estimation, we have by orthogonality
    \begin{align*}
	\|g\|^2 = \|f_0 - a g_p\|^2 + \|X_\perp g_\perp\|^2 + \|g_1\|^2.
    \end{align*}
    By definition,
    \begin{align*}
	\|f_0 - a g_p\|^2 &\le 2 (\|f_0\|^2 + \frac{a_\infty^2}{4} \|v_{0,+} + v_{0,-}\|^2) \\
	&\le 2\|f_0\|^2 + a_\infty^2 (\|v_{0,+}\|^2 + \|v_{0,-}\|^2) \\
	&\le 2\|f_0\|^2 + 4 a_\infty^2 C_P(\|\eta_+ v_{0,+}\|^2 + \|\eta_- v_{0,-}\|^2) \\
	&\le 2\|f_0\|^2 + 4 a_\infty^2 C_P \|f_1\|^2.
    \end{align*}
    Regarding $X_\perp g_s$, we have, using \eqref{eq:property}
    \begin{align*}
	\|X_\perp g_s\|^2  &= 2\|\eta_+ g_s\|^2 \le \|\eta_+ v_{0,+}\|^2 + \|\eta_- v_{0,-}\|^2, 
    \end{align*}
    so that, using \eqref{eq:norms},
    \begin{align*}
	\|X_\perp g_s\|^2 + \|g_{1}\|^2 \le \|f_{1}\|^2. 
    \end{align*}
    Combining all estimates, we obtain
    \begin{align}
	\|g\|^2 \le 2 \|f_0\|^2 + (1+4a_\infty^2 C_P)\|f_1\|^2, 
	\label{eq:estm1}
    \end{align}
    as advertised in \eqref{eq:repest}. We now prove all cases $m\ge 2$ by induction. 


    \noindent For $m \ge 2$ and $f\in L_{(m)}^2(SM)$, write $f = f' + f_m$, with $f'\in L_{(m-1)}^2(SM)$. By Lemma \ref{lem:decomp}, we decompose $f_m$ into
    \begin{align*}
	f_m = X v_{m-1} + w_{m-2} + g_m, 
    \end{align*}
    where $g_m\in \hd_m$, $v_{m-1} \in H_{m-1}$ with $v_{m-1,\pm}\in H_0^1(M)$ and $w_{m-2}\in H_{m-2}$. Then the transport equation 
  \begin{align*}
      Xu + au = -f = - f' - w_{m-2} - g_{m} - X v_{m-1}, 
  \end{align*}
  can be rewritten in the form 
  \begin{align*}
      X(u+v_{m-1}) + a (u+v_{m-1}) = - h,
  \end{align*}
  where
  \begin{align*}
      h := f' + w_{m-2} - a v_{m-1} + g_m \in \bigoplus_{k=0}^{m-1} H_k \oplus \hd_m.
  \end{align*}
  Since $v_{m-1}$ vanishes at the boundary, $h$ satisfies
  \[ I_a h = (u + v_{m-1})|_{\partial_+ SM} = u|_{\partial_+ SM} = I_a f. \]
  We now build the estimate, separating the cases $m=2$ and $m>2$. For $m=2$, we have, by orthogonality of Fourier modes, 
  \begin{align*}
      \|h\|^2 &= \|f_0 + w_0\|^2 + \|f_{1} - a v_{1}\|^2 + \|g_2\|^2.
  \end{align*}
  Bounding each term separately and using the estimates from Lemma \ref{lem:decomp}:
  \begin{align*}
      \|h_0\|^2 + \|h_2\|^2 &= \|f_0 + w_0\|^2 + \|g_2\|^2 \le 2 (\|f_0\|^2 + \|w_0\|^2) + \|g_2\|^2 \le 2 \|f_0\|^2 + 4\|f_2\|^2, \\
      \|h_1\|^2 &= \|f_{1} - a v_{1}\|^2 \le 2(\|f_{1}\|^2 + a_\infty^2 \|v_1\|^2 )\le 2(\|f_{1}\|^2 + 4 a_\infty^2 C_P \|f_2\|^2).
  \end{align*}
  At this point, the component $h_0 + h_1$ does not have the desired form, and we therefore apply the case $m=1$ to it, to obtain the existence of $g' = g_0 + X_\perp g_\perp + g_1$, such that $I_a g' = I_a [h_0 + h_1]$ and, by \eqref{eq:estm1},
  \begin{align*}
      \|g'\|^2 \le 2\|h_0\|^2 + (1+4a_\infty^2 C_P)\|h_1\|^2. 
  \end{align*}
  Now defining $g := g' + g_2$, $g$ has the desired form and we clearly have $I_a g = I_a h = I_a f$. Moreover, combining the last two estimates displayed,
  \begin{align}
      \|g\|^2 &\le 2\|h_0\|^2 + (1+4a_\infty^2 C_P)\|h_1\|^2 + \|g_2\|^2 \\
      &\le 4\|f_0\|^2 + 2(1+4a_\infty^2 C_P) \|f_1\|^2 + 2 ((1+4a_\infty^2 C_P) 4a_\infty^2 C_P+4) \|f_2\|^2. 
  \end{align}
  This in particular satisfies the base case $m=2$ for estimate \eqref{eq:repest}. Suppose now that $m\ge 3$ and that we already decomposed $f_m$ and defined $h$ as in the case $m=2$. We have, by orthogonality
  \begin{align*}
      \|h\|^2 = \|g_m\|^2 + \|f_{m-1} - a v_{m-1}\|^2 + \|f_{m-2} + w_{m-2}\|^2 + \sum_{k=0}^{m-3} \|f_k\|^2,
  \end{align*}
  where, using Lemma \ref{lem:decomp}, the following estimates hold
  \begin{align*}
      \|f_{m-1} - a v_{m-1}\|^2 &\le 2 (\|f_{m-1}\|^2 + a_\infty^2 \|v_{m-1}\|^2) \le 2 (\|f_{m-1}\|^2 + 4a_\infty^2 C_P \|f_{m}\|^2),  \\
      \|f_{m-2} + w_{m-2}\|^2 + \|g_m\|^2 &\le 2 (\|f_{m-2}\|^2 + \|w_{m-2}\|^2) + \|g_m\|^2 \le 2 (\|f_{m-2}\|^2 + \|f_m\|^2),
  \end{align*}
  and for $k\le m-3$, $h_k = f_k$. At this point, the term $h'= \sum_{k=0}^{m-1} h_k$ does not have the desired form, thus we apply the induction step to it, so that there exists $g'\in L^2_{(m-1)}(SM)$ of the desired form, such that $I_a g' = I_a h'$ with estimate 
  \begin{align*}
      \|g'\|^2 \le 4 \sum_{k=0}^{m-3} C^k \|h_k\|^2 + 2 C^{m-2} \|h_{m-2}\|^2 + C^{m-1} \|h_{m-1}\|^2.
  \end{align*}
  Thus upon defining $g = g' + g_m$, we clearly have $I_a g = I_a h = I_a f$, with estimate
  \begin{align*}
      \|g\|^2 = \|g'\|^2 + \|g_m\|^2 &\le 4 \sum_{k=0}^{m-3} C^k \|h_k\|^2 + 2 C^{m-2} \|h_{m-2}\|^2 + C^{m-1} \|h_{m-1}\|^2 + \|g_m\|^2 \\
      &\le 4 \sum_{k=0}^{m-3} C^k \|f_k\|^2 + 2 C^{m-2}\ 2(\|f_{m-2}\|^2 + \|f_m\|^2) \dots\\
      &\qquad \qquad \qquad + C^{m-1} 2(\|f_{m-1}\|^2 +4a_\infty^2 C_P \|f_m\|^2) \\
      &\le 4\sum_{k=0}^{m-2}  C^k \|f_k\|^2 + 2 C^{m-1} \|f_{m-1}\|^2 + C^{m-1} (4/C + 8a_\infty^2 C_P ) \|f_m\|^2,
  \end{align*}
  so that, since $C = 4 + 8a_\infty^2 C_P$, the last term is bounded by $C^m \|f_m\|^2$, and the hypothesis is reconducted. The proof of Theorem \ref{thm:rep} is complete.
\end{proof}

The generalization of Theorem \ref{thm:rep} to integrands with infinite harmonic content with sufficient decay is then immediate. 

\begin{proof}[Proof of Theorem \ref{thm:repinf}] Let $f\in L^{2,C}(SM)$. For $m\ge 1$, let us denote $f^{(m)}:= \sum_{k=0}^m f_k\in L^2_{(m)}(SM)$. Let $g^{(m)}\in L_{(m)}^2(SM)$ as in Theorem \ref{thm:rep} such that $I_a g^{(m)} = I_a f^{(m)}$ satisfying
    \begin{align}
	\|g^{(m)}\|^2 \le 4 \sum_{p=0}^{m} C^p \|f^{(m)}_p\|^2.
	\label{eq:tmp5}
    \end{align}
    By linearity, we have 
    \begin{align*}
	\|g^{(m+n)} - g^{(m)}\|^2 \le 4 \sum_{p=0}^{m+n} C^p \|f^{(m+n)}_p - f^{(m)}_p\|^2 \le 4 \sum_{p=m+1}^{m+n} C^p \|f_p\|^2 \le 4 \sum_{p=m+1}^{\infty} C^p \|f_p\|^2.
    \end{align*}
    Since the right-hand-side converges to zero as $m\to \infty$ regardless of $n$, the sequence $g^{(m)}$ is Cauchy in $L^2(SM)$, thus converges to some $g\in L^2(SM)$. By continuity of $I_a:L^2(SM)\to L^2(\partial_+ SM)$, we have 
    \begin{align*}
	I_a g = \lim_{m\to \infty} I_a g^{(m)} = \lim_{m\to \infty} I_a f^{(m)} = I_a f.
    \end{align*}
    The form \eqref{eq:ginf} of $g$ is inherited from the form \eqref{eq:g} of each $g^{(m)}$ and the fact that each summand belongs to a closed subspace of $L^2(SM)$. Estimate \eqref{eq:repestinf} follows from sending $m\to \infty$ in \eqref{eq:tmp5}.     
\end{proof}

Assuming Theorem \ref{thm:reconstruction}, we now provide a brief proof of Theorem \ref{thm:injectivity}.

\begin{proof}[Proof of Theorem \ref{thm:injectivity}] Let $f\in L^2_{(m)}(SM)$ and define $u$ such that 
  \begin{align*}
    Xu + au = -f, \quad u|_{\partial_- SM} = 0, \qquad u|_{\partial_+ SM} = I_a f = 0.
  \end{align*}  
  Examining the proof of Theorem \ref{thm:rep} closer, we prove that there exists $v$ of degree $m-1$ with components in $H^1_0(M)$, vanishing at $\partial SM$ such that 
  \begin{align*}
    X(u-v) + a(u-v) = -g,
  \end{align*}
  with $g$ as in \eqref{eq:g} such that $I_a g = I_a f = 0$.  From Theorem \ref{thm:reconstruction}, this implies $g=0$. Then $X(u-v) + a(u-v) = 0$ with $u-v$ vanishing at $\partial SM$ implies $u-v = 0$. In particular, $u$ has degree $m-1$, vanishes at $\partial SM$, and satisfies $Xu + au = -f$.   
\end{proof}


\section{Reconstruction. Proof of Theorem \ref{thm:reconstruction}} \label{sec:reconstruction}

\subsection{Data space and boundary operators}\label{sec:boundary}

As introduced in \cite{Monard2015a}, the space $L^2(\partial_+ SM)$ can be given two Hilbert bases 
\begin{align*}
  \B = \{ \phi_{p,q},\ p,q\in \Zm \}, \quad \text{and } \quad \B' = \{ \phi'_{p,q}:= e^{i\alpha} \phi_{p,q},\ p,q\in \Zm \}.
\end{align*}
where we have defined $\phi_{p,q}(\beta,\alpha):= \frac{1}{\pi\sqrt{2}} e^{i(p\beta+2q\alpha)}$.
For further use, we define
\begin{align*}
  u_{p,q} &= \phi_{p,q} + (-1)^p \phi_{p,p-q} = (Id + \SS_A^*) \phi_{p,q}, \\
  v_{p,q} &= \phi_{p,q} - (-1)^p \phi_{p,p-q} = (Id - \SS_A^*) \phi_{p,q}, \\
  u'_{p,q} &= \phi'_{p,q} + (-1)^p \phi'_{p,p-q-1} = (Id + \SS_A^*) \phi'_{p,q}, \\
  v'_{p,q} &= \phi'_{p,q} - (-1)^p \phi'_{p,p-q-1} = (Id - \SS_A^*) \phi'_{p,q},
\end{align*} 
where $\SS_A^* \phi(\beta,\alpha) = \phi(\beta+\pi+2\alpha,-\alpha)$ is the pullback of $\phi$ by the antipodal scattering relation. In particular, we have the splitting $L^2(\partial_+ SM) = \V_+ \oplus \V_-$, where $\V_+ := \ker (Id-\SS_A^*)$ is spanned by either $\{u_{p,q}\}$ or $\{u'_{p,q}\}$, and $\V_- := \ker (Id + \SS_A^*)$ is spanned by either $\{v_{p,q}\}$ or $\{v'_{p,q}\}$. Recall that 
\begin{align}
    P:= A_-^* H A_+ :L^2(\partial_+ SM)\to L^2(\partial_+ SM),    
    \label{eq:P}
\end{align}
splits into $P = P_+ + P_-$ upon defining $P_\pm = A_-^* H_\pm A_+:\V_\pm\to \V_\mp$, and that from \cite{Pestov2004}, Range $I_0 = $ Range $P_-$ and Range $I_\perp =$ Range $P_+$ in smooth topologies. It was then proved in \cite{Monard2015a} that in the $L^2(\partial_+ SM)\to L^2(\partial_+ SM)$ setting, the singular value decompositions of $P_\pm$ makes them roughly $L^2\to L^2$ isometries onto their respective ranges. To be more specific, let us define: 
\begin{align*}
    \V_{+,0} &:= \left\langle u'_{p,q}, \quad 2q > -1, \quad 2p<2q+1 \right\rangle = \text{Range } P_- \\
    \V_{-,\perp} &:= \left\langle v_{p,q}, \quad q\ge 0, \quad p\le q \right\rangle = \text{Range } P_+.
\end{align*}
These subspaces of $L^2(\partial_+ SM)$ capture exactly the modes achieved by $I_0$ and $I_\perp$, and describing the range of these operators there would only require describing rates of decay in the Fourier coefficients. It is also immediate to find that, in the functional setting of \eqref{eq:P}, 
\begin{align}\label{eq:Pstardef}
    P^*_{\pm} = - A_+^* H_\pm A_- : \V_\mp\to \V_\pm.
\end{align}
We now write an explicit right-inverse for the operator $P$ on $\V_{+,0}\oplus \V_{-,\perp}$. In what follows, we recall the definition of the operator $C:= \frac{1}{2}A_-^* H A_-$, which upon splitting $H = H_+ + H_-$, splits accordingly into $C_+:\V_-\to \V_-$ and $C_-:\V_+\to \V_+$.

\begin{lemma}\label{lem:Pdagger} The operators $P_-^\dagger := \frac{1}{4} P_-^*$, $P_+^\dagger:= \frac{1}{4} P_+^* (Id - 12 C_+^2)$ and $P^\dagger:= P_+^\dagger + P_-^\dagger$ are such that $P_-P_-^\dagger$, $P_+P_+^\dagger$ and $PP^\dagger$ are the $L^2(\partial_+ SM)$-orthogonal projections onto $\V_{+,0}$, $\V_{-,\perp}$ and $\V_{+,0} \oplus \V_{-,\perp}$, respectively.
\end{lemma}
In particular, $P^\dagger$ is a right inverse for $P$ on the ranges of $I_0$ and $I_\perp$. 

\begin{proof}
    The following calculations are worked out in \cite[Propositions 1, 2]{Monard2015a}\footnote{Note the typo in \cite{Monard2015a} that the spectral value $i$ of $P_+$ is NOT achieved and should be $-i$ as well.}:
    \begin{align*}
	P_+ u_{p,q} &= -i (\sgn{2q} - \sgn{2(p-q)}) \ v_{p,q} \\
	&= \left\{ \begin{array}{cc}
	    -2i\ v_{p,q} & \text{if } q>0 \text{ and } p<q, \\
	    -i\ v_{p,q} & \text{if } (q>0 \text{ and } p=q) \text{ or } (q=0 \text{ and } p<0), \\
	    0 & \text{otherwise.}
	\end{array}\right. \\
	P_-  v'_{p,q} &= -i (\sgn{2q+1} - \sgn{2p-2q-1})\ u'_{p,q} \\
	&= \left\{ \begin{array}{cc}
	    -2i\ u'_{p,q} & \text{if } q>\frac{-1}{2} \text{ and } p<q+\frac{1}{2}, \\
	    0 & \text{otherwise.}
	\end{array}\right. 
    \end{align*}
    as well as 
    \begin{align*} 
	C_+ v_{p,q} &= \left\{\begin{array}{cc}
	    i\ v_{p,q} & \text{if } q<0 \text{ and } p<q, \\ 
	    -i\ v_{p,q} & \text{if } q>0 \text{ and } p>q, \\
	    \frac{i}{2}\ v_{p,q} & \text{ if } q=0 \text{ and } p<0, \\
	    \frac{-i}{2}\ v_{p,q} & \text{ if } q>0 \text{ and } p=q, \\
	    0 & \text{otherwise}. 
	\end{array}\right. 
    \end{align*} 
    Considering the adjoints 
    \begin{align*}
	P_{+}^* = - A_+^* H_+ A_- : \V_- \to \V_+, \qquad P_-^* = - A_+^* H_- A_- : \V_+ \to \V_-, 
    \end{align*}
    similar calculations allow to establish that 
    \begin{align}
	\begin{split}
	    P_+^* v_{p,q} &= i (\sgn{2q}-\sgn{2(p-q)})\ u_{p,q}, \\
	    P_-^* u'_{p,q} &= i (\sgn{2q+1} - \sgn{2p-2q-1})\ v'_{p,q}.	    
	\end{split}
	\label{eq:Pstar}
    \end{align}
    In particular we obtain
    \begin{align*}
	\frac{1}{4} P_- P_-^* u'_{p,q} &= \frac{1}{4} (\sgn{2q+1} - \sgn{2p-2q-1})^2\  u'_{p,q} \\
	&= \left\{\begin{array}{cc}
	    u'_{p,q} & \text{if } q>\frac{-1}{2} \text{ and } p<q+\frac{1}{2}, \\
	    0 & \text{otherwise,} 
	\end{array}\right.
    \end{align*}
    which means that $\frac{1}{4} P_-^*$ inverts $P_-$ on $\V_{+,0}$. For $P_+$, because of the appearing half spectral values, we need to modify slightly using $C_+$, and arrive at the following 
    \begin{align*}
	\frac{1}{4} P_+ P_+^* (Id - 12 C_+^2) v_{p,q} = \left\{ \begin{array}{cc}
	    v_{p,q} & \text{if } q>0 \text{ and } p<q, \\
	    v_{p,q} & \text{if } (q>0 \text{ and } p=q) \text{ or } (q=0 \text{ and } p<0), \\
	    0 & \text{otherwise.}
	\end{array} \right.
    \end{align*}
    So $\frac{1}{4} P_+^* (Id - 12 C_+^2)$ inverts $P_+$ on $\V_{-,\perp}$. Adding everything together, we see that the operator $P^\dagger = \frac{1}{4} P_-^* + \frac{1}{4} P_+^* (Id - 12 C_+^2) = \frac{1}{4} P^* - 3 P_+^* C_+^2$ achieves what is required. 
\end{proof}

As recorded in the proposition below, these pseudo-inverses are, in fact, the ``filters'' in the filtered-backprojection formulas inverting $I_0$ over $L^2(M)$ and $I_\perp$ over 
\begin{align*}
    \dot{H}^1 (M):= \left\{ u\in H^1(M),\ \int_0^{2\pi} u(e^{i\beta})\ d\beta = 0\right\}.
\end{align*}
The main novelty here is that the inversion formula for solenoidal one-forms allows solenoidal potentials to be supported {\em up to the boundary}. 

\begin{proposition}\label{prop:FBP} With the pseudo-inverses $P_+^\dagger$ and $P_-^\dagger$ defined in Lemma \ref{lem:Pdagger}, we have the reconstruction formulas 
    \begin{align}
	f &= -\frac{1}{2\pi} I_\perp^\sharp P_-^\dagger I_0 f, \qquad f\in L^2(M), \label{eq:rcI0}\\
	h &= \frac{1}{2\pi} I_0^\sharp P_+^\dagger I_\perp h, \qquad h\in \dot{H}^1(M). \label{eq:rcIperp}
    \end{align}     
\end{proposition}

\begin{proof}[Proof of Proposition \ref{prop:FBP}] Equation \eqref{eq:rcI0} is a direct consequence of the formula $f = \frac{1}{8\pi} I_\perp^\sharp A_+^* H A_- I_0 f$ (see \cite{Monard2015a}) and the definition of $P_-^\dagger$. \\
    On to proving \eqref{eq:rcIperp}, by \cite[Lemma 4.3]{Monard2015a}, any $h\in \dot{H}^1(M)$ decomposes into $h_0 + h_\partial$, where $h_0 \in H^1_0(M)$ and $h_\partial = \sum_{k=1}^\infty a_k z^k + b_k \zbar^k$ for some coefficients $a_k, b_k$ with $\sum_{k=1}^\infty (|a_k|^2 + |b_k|^2)(1+k)<\infty$. We already know that for $h_0\in H^1_0(M)$, the following reconstruction formula holds (see, e.g., \cite[Prop. 2.2]{Monard2015})
    \begin{align*}
	h_0 = \frac{-1}{8\pi} I_0^\sharp A_+^* H A_-I_\perp h_0 = \frac{1}{8\pi} I_0^\sharp P_+^* I_\perp h_0.
    \end{align*}
    Surprisingly, the same is only true up to a factor $\frac{1}{4}$ for the term $h_\partial$, as we now show that
    \begin{align}
	h_\partial = \frac{-1}{2\pi} I_0^\sharp A_+^* H A_- I_\perp h_\partial =  \frac{1}{2\pi} I_0^\sharp P_+^* I_\perp h_\partial.
	\label{eq:rcharmonic}	
    \end{align}
    \begin{proof}[Proof of \eqref{eq:rcharmonic}] It is enough to prove it for $h_\partial = z^k$ for any integer $k\ge 1$ and \eqref{eq:rcharmonic} follows by linearity and complex conjugation. Applying the operators one at a time, we first have $I_\perp z^k = -i\pi \sqrt{2} (-1)^k v_{k,k}$ (see, e.g., \cite[Prop. 3]{Monard2015a}). Then by \eqref{eq:Pstar}, we have $P_+^* v_{k,k} = i u_{k,k}$ so that $P_+^* I_\perp z^k = \pi\sqrt{2} (-1)^k u_{k,k}$. Finally by \eqref{eq:backproj}, we have $I_0^\sharp u_{k,k} = 2\pi \left( (u_{k,k})_\psi\right)_0 = (-1)^k \sqrt{2}\ z^k$, so that 
	\begin{align*}
	    \frac{1}{2\pi} I_0^\sharp P_+^* I_\perp z^k = \frac{\pi\sqrt{2} (-1)^k}{2\pi} I_0^\sharp u_{k,k} = \frac{\sqrt{2} (-1)^k}{2} (-1)^k \sqrt{2}\ z^k = z^k, 
	\end{align*} 
	hence \eqref{eq:rcharmonic} is proved. 
    \end{proof}
    With \eqref{eq:rcharmonic} proved, we finally return to proving \eqref{eq:rcIperp}. By virtue of \cite[Prop. 3]{Monard2015a}, we have the relations  $C_+^2 I_\perp h_0 = 0$ and $C_+^2 I_\perp h_\partial = - \frac{1}{4} I_\perp h_\partial$. Combining this with both reconstruction formulas, we deduce that 
    \begin{align*}
	\frac{1}{2\pi} I_0^\sharp P_+^\dagger I_\perp (h_0 + h_\partial) &= \frac{-1}{8\pi} I_0^\sharp A_+^* H A_- I_\perp h_0 + \frac{-1}{8\pi} I_0^\sharp A_+^* H A_- (Id - 12 C_+^2) I_\perp h_\partial \\
	&= \frac{-1}{8\pi} I_0^\sharp A_+^* H A_- I_\perp h_0 + \frac{-1}{2\pi} A_+^* H A_- I_\perp h_\partial \\
	&= h_0 + h_\partial,
    \end{align*}
    hence the proof.
\end{proof}

\subsection{Holomorphic integrating factors} \label{sec:HIF}

As a direct consequence of Proposition \ref{prop:FBP}, we can construct so-called {\em holomorphic integrating factors} for functions and solenoidal one-forms explicitly. 

\begin{proposition}\label{prop:special} Let $P^\dagger$ defined in Lemma \ref{lem:Pdagger}. For any $f_0\in L^2(M)$ and $f_s \in \dot{H}^1(M)$, the function 
    \begin{align*}
	u := -i (Id+iH) (P^\dagger I(f_0 + X_\perp f_s))_\psi,
    \end{align*}
    fiberwise holomorphic by construction, satisfies
    \begin{itemize}
	\item[(i)] $Xu = -f_0 -X_\perp f_s$
	\item[(ii)] $u_0 = -i f_s$.   
    \end{itemize}
\end{proposition}

\begin{proof}[Proof of Proposition \ref{prop:special}] Proving claim $(i)$ amounts to computing 
    \begin{align*}
	u_0 &= -i \left( (Id+iH) (P^\dagger I(f_0 + X_\perp f_s))_\psi \right)_0 \\
	&= -i \left( (P^\dagger I(f_0 + X_\perp f_s))_\psi \right)_0 = \frac{-i}{2\pi} I_0^\sharp P^\dagger I(f_0 + X_\perp f_s) = -if_s.
    \end{align*}
    Proving claim $(ii)$ amounts to computing
    \begin{align*}
	Xu &= - X i(Id + iH) (P^\dagger I(f_0 + X_\perp f_s) )_\psi \\
	&= XH (P^\dagger I(f_0 + X_\perp f_s) )_\psi \\
	&= - [H,X] (P^\dagger I(f_0 + X_\perp f_s) )_\psi \\
	&= - \left( X_\perp (P^\dagger I(f_0 + X_\perp f_s)) )_\psi  \right)_0 - X_\perp \left( (P^\dagger I(f_0 + X_\perp f_s) )_\psi \right)_0 \\
	&=  \frac{1}{2\pi} I_\perp^\sharp P^\dagger I(f_0 + X_\perp f_s) -\frac{1}{2\pi} X_\perp I_0^\sharp P^\dagger I(f_0 + X_\perp f_s) \\
	&= - f_0 - X_\perp f_s,
    \end{align*}
    hence the result.
\end{proof}

While these will be used to construct a holomorphization operator of transport solutions in Section \ref{sec:g0gs}, integrating factors for attenuation $a$ will be used at several places throughout. 

\todo[inline]{Show that $w_a\in L^{2,X}(SM)$}

\begin{proposition}\label{prop:HIF}
    For a function $a\in \C^0(M,\Cm)$, the function defined on $SM$ by
    \begin{align}
	w_a = 2\pi i (Id + iH) n_\psi, \qquad n = -\frac{1}{8\pi} P_-^* I_0 a \in \V_-
	\label{eq:wa}
    \end{align}
    is a fiberwise odd, holomorphic solution of $Xw_a = -a$, whose restrictions to $\partial_+ SM$ and $\partial SM$  are given by
    \begin{align}
	\rho_a := w_a|_{\partial_+ SM} = \frac{1}{2} I_0 a - \frac{i}{4} P_-^* I_0 a, \qquad w_a|_{\partial SM} = \frac{1}{2} A_- I_0 a - \frac{i}{4} A_+ P_-^* I_0 a. \label{eq:rhoa}
    \end{align}    
\end{proposition}

\begin{proof} That $w_a$ is holomorphic is immediate and $w_a$ solves $Xw_a = -a$ as a consequence of Proposition \ref{prop:special}, and since $n\in \V_-$, $n_\psi$ is odd. We now compute, using that $(n_\psi)|_{\partial SM} = A_+ n$, 
    \begin{align*}
	\rho_a := w_a|_{\partial_+ SM} &= 2\pi i (Id + iH) (n_\psi)|_{\partial SM}|_{\partial_+ SM} \\
	&= -\frac{i}{4} (Id + iH) A_+ P_-^* I_0 a|_{\partial_+ SM} \\
	&= -\frac{i}{8} (A_+^* + A_-^*)(Id + iH) A_+ P_-^* I_0 a.
    \end{align*}
    Symmetry considerations give that $P_-^* I_0 a \in \V_-$. The sum above splits into four terms, whose leftmost factors are $A_-^* A_+ \equiv 0$, $A_+^*A_+ = 2Id$ and $A_-^* H A_+ = P_-$ (when acting on $\V_-$) and $A_+^* H A_+$, leading up to 
    \begin{align*}
	\rho_a = -\frac{i}{4} P_-^* I_0 a + \frac{1}{8} P_- P_-^* I_0 a + \frac{1}{8} A_+^* H A_+ P_-^* I_0 a. 
    \end{align*}
    By Lemma \ref{lem:Pdagger}, we see that the second term in the right-hand side equals $\frac{1}{2} I_0 a$. On to the last term, with the fact that 
    \begin{align*}
	A_+^* H A_+ v'_{p,q} = (-i) (\sgn{2q+1} + \sgn{2p-2q-1})v'_{p,q},
    \end{align*}
    we deduce that 
    \begin{align*}
	A_+^* H A_+ P_-^* u'_{p,q} &= i (\sgn{2q+1} - \sgn{2p-2q-1}) A_+^* H A_+ v'_{p,q} \\
	&= (\sgn{2q+1}^2 - \sgn{2p-2q-1}^2) v'_{p,q} \\
	&= 0,  
    \end{align*}
    hence the formula for $\rho_a$ in \eqref{eq:rhoa} holds. In addition, since $w_a$ is fiberwise odd, then $w_a|_{\partial SM}$ is the fiberwise odd extension of $\rho_a$ to $\partial SM$, and since $I_0 a \in \V_+$ and $P_-^* I_0 a \in \V_-$, this is equivalent to writing
    \begin{align*}
	w_a|_{\partial SM} = \frac{1}{2} A_- I_0 a - \frac{i}{4} A_+ P_-^* I_0 a. 
    \end{align*}
    The proof is complete.
\end{proof}

\subsection{Invariant distributions with prescribed harmonic moments} \label{sec:invariant}

The present section aims at producing fiberwise holomorphic invariant distributions $h_\psi$ (as in \eqref{eq:hpsi}) with fiberwise average $(h_\psi)_0\in L^2(\ker\dbar)$. Since $h\mapsto (h_\psi)_0$ is the adjoint of the ray transform in the $L^2(M)\to L^2(\partial_+ SM, \cos\alpha)$ setting, this can also be formulated as a surjectivity statement for this adjoint, as was initially done in \cite[Theorem 1.4]{Pestov2005} for smooth topologies and simple Riemannian surfaces. The main difference here is that the target space in the ray transform is a different one and, while leading to more explicit constructions (see Theorem \ref{thm:I0star} below), it would not be amenable to the argument in \cite{Pestov2005} since in the present setting, the normal operator $I_0^* I_0$ associated with the restriction $I_0:H_0\to L^2(\partial_+ SM)$ is not an elliptic pseudo-differential operator after being extended to a slightly larger domain. In the present setting, a direct calculation using Santal\'o's formula leads to the expression 
\begin{align}
    I_0^* h = \left( \left( \frac{h}{\cos\alpha} \right)_\psi\right)_0, \qquad I_0^*: L^2(\partial_+ SM) \to H_0.
    \label{eqLI0star}
\end{align}
Here we first aim at finding explicit preimages by $I_0^*$ of elements in $L^2(\ker \dbar)$, and in the case of the Euclidean unit disc, this is again rather explicit, by directly exhibiting the singular value decomposition of $I_0\circ \iota: L^2(\ker \dbar) \to L^2(\partial_+ SM)$, where $\iota:L^2(\ker \dbar)\to H_0$ is the inclusion map. 


\todo[inline]{Explain $I_0$ denotes the transform $I$ restricted to $H_0$ \\ Beurling transform in Astala-Iwaniec p150.}

\paragraph{Singular value decomposition (SVD) of $I_0 \circ \iota$.} We first define
\begin{align}
    Z_k (\x) = \sqrt{\frac{k+1}{2\pi^2}} (x+iy)^k, \qquad k=0,1,\dots , 
    \label{eq:Zk}
\end{align}
normalized so that $\|Z_k\|_{L^2(SM)}=1$. In particular, by construction,
\begin{align*}
    L^2(\ker \dbar) = \ell^2\ \left( \{ Z_k \}_{k=0}^\infty \right).
\end{align*}
In addition, it is computed easily (see e.g., \cite[pp449-450]{Monard2015a}), that 
\begin{align*}
    I_0 Z_k (\beta, \alpha) = \frac{(-1)^k}{\sqrt{2\pi^2} \sqrt{k+1}} e^{ik\beta} (e^{i(2k+1)\alpha} + (-1)^k e^{-i\alpha}) = \frac{(-1)^k}{\sqrt{k+1}}\ u'_{k,k},
\end{align*}
so that 
\begin{align*}
    I_0 (L^2(\ker \dbar)) = I_0 \left( \ell^2 \left( \left\{ Z_k \right\}_{k=0}^\infty \right)  \right) = h^{\frac{1}{2}} \left( \left\{ u'_{k,k} \right\}_{k=0}^\infty \right).
\end{align*}
Moreover, since $\dprod{u'_{k,k}}{u'_{n,n}}_{\partial_+ SM} = 2 \delta_{kn}$, we obtain directly that the SVD of $I_0\circ\iota: L^2(\ker \dbar) \to \ell^2 \left( \left\{ u'_{k,k} \right\}_{k=0}^\infty \right)$, is given by 
\begin{align*}
    \left(Z_k,\ \frac{(-1)^k}{\sqrt{2}} u'_{k,k},\ \sqrt{\frac{2}{k+1}} \right), \qquad k=0,1,\dots
\end{align*}
This implies in particular that the SVD of $(I_0\circ\iota)^* = \iota^*\ I_0^*:  \ell^2 \left( \left\{ u'_{k,k} \right\}_{k=0}^\infty \right) \to L^2(\ker \dbar)$ is nothing but 
\begin{align*}
    \left( \frac{(-1)^k}{\sqrt{2}} u'_{k,k},\ Z_k,\ \sqrt{\frac{2}{k+1}}  \right), \qquad k=0,1,\dots,
\end{align*}
or, in other words
\begin{align*}
    \iota^* I_0^* \left( \frac{(-1)^k}{\sqrt{2}} u'_{k,k} \right) =  \sqrt{\frac{2}{k+1}} \ Z_k, \qquad \Leftrightarrow \qquad \iota^* I_0^* \left( (-1)^k \frac{\sqrt{k+1}}{2} u'_{k,k} \right) = Z_k,
\end{align*}
where $\iota^*$ is the $L^2(SM)$-orthogonal projection onto $L^2(\ker \dbar)$. For the statement above to express the existence of invariant distributions with prescribed average, it is now absolutely necessary to remove $\iota^*$ from the equalities above, and make it a surjectivity result for $I_0^*$ and not $\iota^* I_0^\star$. To this end, we must go through the following direct calculation, whose proof is relegated to the Appendix.  
\begin{proposition}\label{prop:Wk} For any integer $k\ge 0$, we have 
    \begin{align}
	(-1)^k \frac{\sqrt{2}}{\pi} z^k = \left( \left( \frac{u'_{k,k}}{\cos \alpha} \right)_\psi\right)_0 = 2\left( \left( u_{k,k} \right)_\psi\right)_0.
	\label{eq:backproj}
    \end{align}
\end{proposition}
Now defining 
\begin{align}
    W_k := (-1)^k \frac{\sqrt{k+1}}{2}\ u'_{k,k} = (-1)^k \frac{\sqrt{k+1}}{2\pi\sqrt{2}} e^{ik\beta} (e^{i(2k+1)\alpha} + (-1)^k e^{-i\alpha}), \quad k=0,1,\dots,
    \label{eq:Wk}
\end{align}
and with $Z_k$ as in \eqref{eq:Zk}, Proposition \ref{prop:Wk} indeed implies
\begin{align*}
    I_0^* W_k = \left(\left(\frac{W_k}{\cos \alpha}\right)_\psi\right)_0 = Z_k, \qquad k=0,1,\dots
\end{align*}
This motivates the definition of the kernel $\G$, based on the series $\sum_{k=0}^\infty (k+1)\zeta^k = (1-\zeta)^{-2}$, convergent for $|\zeta|<1$: 
\begin{align}
    \begin{split}
	\G(z; \beta,\alpha) &:= \sum_{k=0}^\infty \overline{W_k}(\beta,\alpha) Z_k(z) \\
	&= \frac{1}{4\pi^2} \sum_{k=0}^\infty (k+1) (-ze^{-i\beta})^k (e^{-i(2k+1)\alpha} + (-1)^k e^{i\alpha}) \\
	&= \frac{1}{4\pi^2} \left( \frac{e^{-i\alpha}}{(1+ze^{-i(\beta+2\alpha)})^2} + \frac{e^{i\alpha}}{(1-z e^{-i\beta})^2} \right).
    \end{split}    
    \label{eq:Green}
\end{align}
Based on the property that $I_0^* W_k = Z_k$, we are able to fomulate the following
\begin{theorem}\label{thm:I0star}
    The operator $I_0^*:h^{-\frac{1}{2}} \left( \{u'_{k,k}\}_{k=0}^\infty \right)\to L^2(\ker \dbar)$ is surjective. More specifically, for any $f\in L^2(\ker \dbar)$, given by $f = \sum_{k=0}^\infty \dprod{f}{Z_k}_{SM} Z_k$, the function $W_f\in h^{-\frac{1}{2}} \left( \{u'_{k,k}\}_{k=0}^\infty \right)$ given by  
    \begin{align}
	W_f := \sum_{k=0}^\infty \dprod{f}{Z_k}_{SM} W_k, \quad \text{satisfies} \qquad I_0^* W_f = \left( \left( \frac{W_f}{\cos\alpha} \right)_\psi \right)_0 = f. 
	\label{eq:Wf}
    \end{align}
    Moreover, the distribution $\left( \frac{W_f}{\cos\alpha} \right)_\psi \in h^{-\frac{1}{2}} \left( \left( \frac{u'_{k,k}}{\cos\alpha} \right)_\psi, k=0\dots\infty \right)$ 
    \begin{itemize}
	\item[$(i)$] is fiberwise holomorphic, 
	\item[$(ii)$] satisfies $\left\langle \left( \frac{W_f}{\cos\alpha} \right)_\psi, e^{im\theta}Z_k \right\rangle_{SM} = 0$ for every $k\ge 1$ and $m>0$. 
    \end{itemize}
\end{theorem}

\todo[inline]{Study the harmonic content of $(u'_{p,q}/\cos\alpha)_\psi$. Show that multiplication by a $C^0$ does not alter the distribution. }


\begin{remark} Using \eqref{eq:Green}, $W_f$ defined in \eqref{eq:Wf} also takes the integral representation
    \begin{align*}
	W_f(\beta,\alpha) = 2\pi \int_{M} f(\x)\ \overline{\G (x+iy; \beta,\alpha)}\ d^2\x, \qquad (\beta,\alpha)\in \partial_+ SM,
    \end{align*}
    where the $2\pi$ factor comes from the fact that the initial integral is over $SM$. 
\end{remark}

\begin{remark} As stated in Theorem \ref{thm:I0star}, $\left(\frac{W_f}{\cos\alpha}\right)_\psi$ makes sense as an element in the $h^{-\frac{1}{2}}$ span of $\left( \frac{u'_{k,k}}{\cos\alpha} \right)_\psi$, an orthogonal family of norm $2$ in $L^2(SM)$, as is readily seen by
    \begin{align*}
	\left\langle \left( \frac{u'_{k,k}}{\cos\alpha} \right)_\psi, \left( \frac{u'_{\ell,\ell}}{\cos\alpha} \right)_\psi\right\rangle_{SM} &= \int_{\partial_+ SM} \int_0^{2\cos\alpha} \left( \frac{u'_{k,k}}{\cos\alpha} \frac{u'_{\ell,\ell}}{\cos\alpha} \right)_\psi (\phi_t(\beta,\alpha))\ dt\ \cos\alpha\ d\alpha\ d\beta \\
	&= 2 \left\langle u'_{k,k}, u'_{\ell,\ell} \right\rangle_{\partial_+ SM} = 4\delta_{k,l}. 
    \end{align*}
\end{remark}


\begin{proof}[Proof of Theorem \ref{thm:I0star}] 
    The first statement is a direct consequence of Proposition \ref{prop:Wk}, and we now prove $(i)$ and $(ii)$. \\
    \noindent {\bf Proof of $(i)$.} In light of Lemma \ref{lem:holo} below, it is enough to show that $A_+ \left( \frac{W_k}{\cos\alpha} \right)$ is holomorphic on the fibers of $\partial SM$, which means that its harmonic content in $\alpha$ only consists of nonnegative harmonics. In order to check, it should be noted that $\frac{W_k}{\cos\alpha}$, initially defined on $\partial_+ SM$, belongs to $\V_+$, which means that extending it to $\partial_- SM$ by evenness is the same as extending by evenness w.r.t. $\alpha\mapsto \alpha+\pi$, and judging by the given expression, this simply consists of extending the expression we already have to $\partial SM$. Then the calculation \eqref{eq:firstcal} applies to $A_+ \left(\frac{u'_{k,k}}{\cos\alpha}\right)$ on the whole of $\partial SM$, i.e. 
    \begin{align*}
	A_+ \left(\frac{u'_{k,k}}{\cos\alpha}\right) = \frac{\sqrt{2}}{\pi} e^{ik\beta} (-1)^k \sum_{p=0}^k (-1)^p e^{2ip\alpha}, \qquad (\beta,\alpha) \in \partial SM, 
    \end{align*}
    hence $A_+\left(\frac{W_k}{\cos\alpha}\right) = (-1)^k \frac{\sqrt{k+1}}{2} A_+\left(\frac{u'_{k,k}}{\cos\alpha}\right)$ only has nonnegative Fourier modes in $\alpha$. Since in addition $\left(\left( \frac{W_f}{\cos\alpha} \right)_\psi \right)_0= f\in L^2(\ker \dbar)$, by virtue of Lemma \ref{lem:holo}, $\left( \frac{W_f}{\cos\alpha} \right)_\psi$ is fiber-holomorphic on $SM$. \\
    {\bf Proof of $(ii)$.} It suffices to show that for every $p\ge 0$, 
    \begin{align*}
	\left\langle \left( \frac{W_p}{\cos\alpha} \right)_\psi, e^{im\theta}Z_k \right\rangle_{SM} = 0, \qquad k\ge 0, \quad m>0.
    \end{align*}
    For $m$ odd, this is already clear because then $e^{im\theta} Z_k$ is fiberwise odd, while, since $\frac{W_k}{\cos\alpha}\in \V_+$, $\left( \frac{W_k}{\cos\alpha} \right)_\psi$ is fiberwise even. Therefore it remains to check orthogonality for $m = 2q$ even. In this case, the integration by parts formula \eqref{eq:IBPatt} with $a=0$ reads
    \begin{align*}
	\dprod{e^{i2q\theta} Z_k}{\left( W_p/\cos\alpha \right)_\psi}_{L^2(SM)} = \dprod{I[e^{i2q\theta} Z_k]}{W_p}_{L^2(\partial_+ SM)}.
    \end{align*}
    Now using \cite[Proposition 4]{Monard2015a}, $I[e^{i2q\theta} Z_k]$ is proportional to $u'_{2q+k,q+k}$, and since $W_p$ is a multiple of $u'_{p,p}$, their inner product always vanishes. $(ii)$ is proved. 
\end{proof}

\begin{lemma}\label{lem:holo} Suppose $h\in \V_+$ is such that $A_+ h$ is fiberwise holomorphic on $\partial SM$ and $(h_\psi)_0 \in \ker \dbar$. Then $h_\psi$ is fiber-holomorphic on $SM$.     
\end{lemma}

\begin{proof}[Proof of Lemma \ref{lem:holo}]
    We must show that $v := (Id - iH) h_\psi = (h_\psi)_0$. The function $v$ satisfies 
    \begin{align*}
	X v = X(Id -iH) h_\psi = i[H,X] h_\psi = i X_\perp (h_\psi)_0 + \cancel{i (X_\perp h_\psi)_0},
    \end{align*}
    where the last term vanishes because $h\in \V_+$ (thus $h_\psi$ is fiberwise even, thus $X_\perp h_\psi$ is fiberwise odd). Since $(h_\psi)_0\in \ker \dbar$ and since $e^{-i\theta} \dbar = \frac{1}{2} (X- iX_\perp)$, then $i X_\perp (h_\psi)_0 = X(h_\psi)_0$, so the equation above can be rewritten as
    \begin{align}
	X( v - (h_\psi)_0) = 0.
	\label{eq:transvv}
    \end{align}
    At the boundary, 
    \[ (h_\psi)_0|_{\partial SM} = (A_+ h)_0 = (Id-iH) A_+ h = v|_{\partial SM}, \]
    so $(v - (h_\psi)_0)|_{\partial SM} = 0$, then by \eqref{eq:transvv}, $v = (h_\psi)_0$ on $SM$. Lemma \ref{lem:holo} is proved. 
\end{proof}

\todo[inline]{Formulated as it is, Lemma \ref{lem:holo} holds on any orientable Riemannian surface with boundary $(M,g)$, with $\{X,X_\perp,V\}$ the canonical framing of its unit tangent bundle $SM$. }

\subsection{Reconstruction of $g_m\in \hd_m$} \label{sec:gm}

We now explain how to reconstruct $g_m\in \hd_m$ from knowledge of $I_a(g'+g_m)$, where, if $m\ge 2$, $g'\in L^2_{(m-1)}(SM)$ and if $m=1$, $g' = g_0 + X_\perp g_s$ where $g_0\in L^2(M)$ and $g_s\in H^1_0(M)$. Recall that $g_m$ can be written as
\[ g_{m}(\x,\theta) = e^{im\theta} g_{m,+}(\x)  + e^{-im\theta} g_{m,-}(\x),\quad g_{m,+}\in L^2(\ker\dbar),\quad g_{m,-} \in L^2(\ker \partial). \]

\paragraph{Reconstruction of $e^{im\theta} g_{m,+} \in \hd_{m,+}$.} Since $\hd_{m,+} = \{e^{im\theta}f(\x),\ f\in L^2( \ker\dbar)\}$, it is clear that a Hilbert orthonormal basis of $\hd_{m,+}$ is $\{e^{im\theta} Z_k\}_{k=0}^\infty$ with $Z_k$ defined in \eqref{eq:Zk}. By Parseval's, we then have  
\begin{align}
    e^{im\theta} g_{m,+} = \sum_{k=0}^\infty \dprod{e^{im\theta} g_{m,+}}{e^{im\theta} Z_k}_{SM} e^{im\theta} Z_k.
    \label{eq:fmrep}
\end{align}
We now explain how to recover the inner products above from known data. Choose as integrating factor $\overline{w_{\overline{a}}}$, defined via Proposition \ref{prop:HIF}, that is, $\overline{w_{\overline{a}}}$ is a fiberwise antiholomorphic, odd solution of $Xw = -a$, with $\overline{w_{\overline{a}}}|_{\partial_+ SM} = \overline{\rho_{\overline{a}}}$. Then the integration by parts formula \eqref{eq:IBPatt} with integrand $g'+g_m$, integrating factor $\overline{w_{\overline{a}}}$ and invariant distribution $\phi = e^{im\theta} \left( \frac{W_k}{\cos\alpha} \right)_\psi = \left( e^{im(\beta+\alpha+\pi)} \frac{W_k}{\cos\alpha} \right)_\psi$ reads:
\begin{align*}
    \dprod{e^{-\overline{w_{\overline{a}}}}(g'+g_m)}{\phi}_{SM} = \dprod{e^{-\overline{\rho_{\abar}}} I_a (g'+g_m)}{(-1)^m e^{im(\beta+\alpha)} W_k}_{\partial_+ SM}.
\end{align*}
The crucial observation now comes from simplifying the left-hand side above by considerations of harmonic content: $e^{-\overline{w_{\overline{a}}}}$ is antiholomorphic and of the form $e^{-\overline{w_{\overline{a}}}} = 1 + \O_{<0}$, so if $m\ge 2$, $e^{-\overline{w_{\overline{a}}}}(g + g_m) = e^{im\theta} g_{m,+} + \O_{<m}$ and $\phi = e^{im\theta} Z_k + \O_{>m}$, so the left-hand side in the last equation simplifies into
\begin{align*}
    \dprod{e^{-\overline{w_{\overline{a}}}}(g'+g_m)}{\phi}_{SM} = \dprod{e^{im\theta} g_{m,+}}{e^{im\theta} Z_k}_{SM}.
\end{align*}
The result is still true when $m=1$ and $g' = g_0 + X_\perp g_s = g_0 - \frac{1}{i} e^{-i\theta} \dbar g_s + \frac{1}{i} e^{i\theta} \partial g_s$, in which case the only potentially troublesome term $\dprod{e^{i\theta} \partial g_s}{e^{i\theta}Z_k}_{SM}$ is still zero, as can be seen by integrating by parts on $M$. 

Combining these observations with \eqref{eq:fmrep}, we obtain a reconstruction formula for $g_{m,+}$ given by 
\begin{align}
    g_{m,+} = \sum_{k=0}^\infty \dprod{e^{-\overline{\rho_{\abar}}} I_a (g'+g_m)}{(-1)^m e^{im(\beta+\alpha)}  W_k}_{\partial_+ SM} Z_k.
    \label{eq:fmrecons}
\end{align}
Using the kernel $\G$ defined in \eqref{eq:Green}, we can also write the following integral representation: 
\begin{align}
    g_{m,+} (\x) = (-1)^m \int_{\partial_+ SM} e^{-\overline{\rho_{\abar}}} I_a (g'+g_m) (\beta,\alpha) e^{-im (\beta+\alpha)} \G (x+iy; \beta,\alpha) \ d\beta\ d\alpha.
    \label{eq:fmrecons_int}
\end{align}

\begin{remark}
    In the case of zero attenuation, writing $f_m(\x,\theta) = g_{m,+}(\x)e^{im\theta}$, the reconstruction formula \eqref{eq:fmrecons_int} gives 
    \begin{align*}
	f_m (\x,\theta) = (-1)^m e^{im\theta} \int_{\partial_+ SM} I f_m (\beta,\alpha) e^{-im (\beta+\alpha)} \G (x+iy; \beta,\alpha) \ d\beta\ d\alpha.
    \end{align*}
    With the notation of the current paper, we then recover, via a different way, the formulas presented in \cite[Theorem 2.4]{Monard2015a}.
\end{remark}

\paragraph{Reconstruction of $e^{-im\theta} g_{m,-}\in \hd_{m,-}$.} We have the following obvious identity
\begin{align*}
    \overline{I_a (g'+g_m)} = I_{\overline{a}} (\overline{g'} + \overline{g_m}),
\end{align*}
where $\overline{g'}\in L^2_{(m-1)}(SM)$ and $\overline{g_m}\in \hd_m(M)$, with decomposition 
\begin{align*}
    \overline{g_m} = e^{-im\theta} \overline{g_{m,+}} + e^{im\theta} \overline{g_{m,-}}. 
\end{align*}
Using the formulas above, we can then reconstruct $\overline{g_{m,-}}$ from $I_{\overline{a}} (\overline{g'} + \overline{g_m})$ via \eqref{eq:fmrecons_int} {\it mutatis mutandis}: 
\begin{align*}
    \overline{g_{m,-}} (\x) = (-1)^m \int_{\partial_+ SM} e^{-\overline{\rho_{a}}} I_{\overline{a}} (\overline{g'} + \overline{g_m}) (\beta,\alpha)  e^{-im (\beta+\alpha)} \G (x+iy; \beta,\alpha) \ d\beta\ d\alpha.
\end{align*}
Complex-conjugating, we arrive at: 
\begin{align}
    g_{m,-}(\x) = (-1)^{m} \int_{\partial_+ SM} e^{-\rho_{a}} I_a (g'+g_m)(\beta,\alpha)  e^{im (\beta+\alpha)} \overline{\G (x+iy; \beta,\alpha)} \ d\beta\ d\alpha.
    \label{eq:fmrecons_int3}
\end{align}

\begin{remark}[Speed-up of formula \eqref{eq:fmrecons_int}]
    As it stands, formula \eqref{eq:fmrecons_int} can be sped up, noticing that $\G(z; \beta,\alpha) = \frac{Id+\SS_A^*}{2} \left( \frac{e^{i\alpha}}{2\pi^2 (1-ze^{-i\beta})^2} \right)$. Since $(Id + \SS_A^*)/2$ is self-adjoint, equation \eqref{eq:fmrecons_int} becomes
    \begin{align*}
	g_{m,+} (\x) &= \frac{(-1)^m}{2\pi^2} \int_{\partial_+ SM} \frac{Id + \SS_A^*}{2} \left( e^{-\overline{\rho_{\abar}}} I_a (g'+g_m) e^{-im (\beta+\alpha)} \right) \frac{e^{i\alpha}}{(1-(x+iy)e^{-i\beta})^2} \ d\beta\ d\alpha, \\
	&= \frac{(-1)^m}{2\pi^2} \int_0^{2\pi} \frac{e^{-im\beta}}{(1-(x+iy)e^{-i\beta})^2} \int_{-\frac{\pi}{2}}^{\frac{\pi}{2}} \frac{Id + \SS_A^*}{2} \left( e^{-\overline{\rho_{\abar}}} I_a (g'+g_m)   e^{i(-m+1)\alpha} \right)\ d\alpha\ d\beta,
    \end{align*}
    (similarly for \eqref{eq:fmrecons_int3}), where the integral in $\alpha$ does not depend on $\x$. This is significantly faster than \eqref{eq:fmrecons_int}, which requires integrating against the non-separable kernel $\G(z; \beta,\alpha)$. 
\end{remark}

\todo[inline]{Write down the real-valued case.}

\subsection{Reconstruction of $g_0$ and $g_s$} \label{sec:g0gs}

The approach above showed that we can reconstruct the ``residual terms'' first, from highest order to lowest. After their forward transform is successively removed from the data, we are them left with reconstructing $(g_0,g_s)$ from $I_a [g_0 + X_\perp g_s]$, which is the purpose of this section. Such an inversion method was first proposed in the context of simple surfaces in \cite{Assylbekov2017}, though the present Euclidean case allows for even more expliciteness, and we will repeat the arguments here for completeness. With the right inverse $P^\dagger$ of $P$ constructed in Section \ref{sec:boundary}, we first construct a so-called {\em holomorphization} operator, adapting \cite[Proposition 6.1]{Assylbekov2017}. Define the operator $\Bh :L^2(SM) \to L^2(\partial_+ SM)$\footnote{It is also proved in \cite{Assylbekov2017} to be a mapping $\Bh:C^\infty(SM)\to C^\infty(\partial_+ SM)$.} by
\begin{align}
    \Bh h := \frac{1}{2} \left[ (Id - iH) h + i (Id + iH) (A_+ P^\dagger A_-^\star (Id-iH)h) \right]|_{\partial_+ SM}, \qquad h\in L^2(\partial SM).
    \label{eq:B}
\end{align}
Then we have the following

\begin{proposition}\label{prop:holomorphization} The operator $\Bh$ defined in \eqref{eq:B} satisfies: for any $f\in C^\infty(SM)$ and any smooth solution $u$ of $Xu = -f$, the function $\uh := u - (\Bh (u|_{\partial SM}))_\psi$ satisfies $X\uh = -f$ and 
    \begin{enumerate}
	\item If $f = \O_{\ge -1}$, then $\uh$ is holomorphic.
	\item If, additionally $f_{-1} = 0$, then $\uh_0=0$. 
    \end{enumerate}
    Similarly, defining $\Ba h := \overline{\Bh (\overline h)}$, the function $\ua := u - (\Ba (u|_{\partial SM}))_\psi$ satisfies $X\ua = -f$ and 
    \begin{enumerate}
	\item[1'.] If $f = \O_{\le 1}$, then $\ua$ is anti-holomorphic.
	\item[2'.] If, additionally $f_{1} = 0$, then $\ua_0 =0$. 
    \end{enumerate}    
\end{proposition}

\begin{proof}[Proof of Proposition \ref{prop:holomorphization}] 
    Let $f\in C^\infty(SM)$ and $u$ a solution of $Xu = -f$. Write 
    \begin{align*}
	u = \frac{1}{2} (u^{(+)} + u^{(-)}), \qquad u^{(\pm)} := (Id \pm iH) u,
    \end{align*}
    where $(Id - iH) u$ solves the PDE
    \begin{align*}
	X(Id - iH) u = (Id-iH) Xu + [X,Id-iH] u = -2f_{-1} - f_0 + i (X_\perp u)_0 + i X_\perp u_0.
    \end{align*}
    Applying the Hodge decomposition to the one-form $2f_{-1}$, there exists $g\in H^1_0(M)$ and $h\in \dot{H}^1(M)$ such that $2f_{-1} = Xg + X_\perp h$, in which case the previous equation can be rewritten as
    \begin{align*}
	X (u^{(-)} + g) = - (f_0 - i(X_\perp u)_0) - X_\perp (h - iu_0). 
    \end{align*}
    Upon integrating along each line, we make appear
    \begin{align}
	A_-^* (u^{(-)}|_{\partial SM}) = A_-^* (u^{(-)} + g)|_{\partial SM} = I ( (f_0 - i (X_\perp u)_0) + X_{\perp} (h-iu_0) ). 
	\label{eq:Amstar}
    \end{align}
    This motivate that we define 
    \[ u' = - i(Id + iH) (P^\dagger A_-^* (u^{(-)}|_{\partial SM}) )_\psi = - i(Id + iH) (P^\dagger I ( (f_0 - i (X_\perp u)_0) + X_{\perp} (h-iu_0) ).  \]
    $u'$ is holomorphic and by virtue of Proposition \ref{prop:special}, $u'$ also solves
    \begin{align*}
	Xu' = - (f_0 - i (X_\perp u)_0) - X_\perp (h-iu_0).
    \end{align*}
    and $u'_0 = -ih -u_0$. We then rewrite $u$ as 
    \begin{align*}
	u = \frac{1}{2} (u^{(+)} - g + u') + \frac{1}{2} (u^{(-)} + g - u'),
    \end{align*}
    where the first summand $\uh := \frac{1}{2} (u^{(+)} - g + u')$ is holomorphic, and where the second summand satisfies $X \left( \frac{1}{2}(u^{(-)} + g - u')\right) = 0$, so that it is equal to some $h_\psi$, where $h = \frac{1}{2} (u^{(-)} + g - u')|_{\partial_+ SM} = \Bh (u|_{\partial SM})$ by construction. Therefore, Claim 1 is proved. As for Claim 2, if $f_{-1} = 0$, then the Hodge decomposition above becomes $h = g = 0$, and we read
    \[ 2 \uh_0 = (u^{(+)} - g + u')_0 = u_0 - g - ih - u_0 = 0. \]
    Thus Proposition \ref{prop:holomorphization} is proved. 
\end{proof}

Out of the holomorphization operator, we are able to derive reconstruction formulas for $g_0$ and $g_s$. Such formulas were derived in \cite{Assylbekov2017} and we repeat the proof here for completeness. 

\begin{theorem}[Reconstruction of $(g_0,g_s)$]\label{thm:inversionfh0} Let $a\in C^\infty(\Dm)$. Define $w_a$ and $\overline{w_{\overline{a}}}$ following Eq. \eqref{eq:wa}, $\rho_a:= w_a|_{\partial_+ SM}$ and $\overline{\rho_{\overline{a}}}:= \overline{w_{\overline{a}}}|_{\partial_+ SM}$, and let $\Bh$ and $\Ba$ as above. Then the functions $(g_s,g_0)\in H^1_0(M)\times L^2(M)$ can be reconstructed from data $\I := I_{a} (g_0 + X_\perp g_s)$ (extended by zero on $\partial_- SM$) via the following formulas:
    \begin{align*}
	g_0 &= -\eta_+ \Dh_{-1} - \eta_- \Da_1 - \frac{a}{2} \left( \Dh_0 + \Da_0 + i(g_+ - g_-) \right), \\
	g_s &= \frac{1}{2} (g_+ + g_-) - \frac{i}{2} (\Dh_0 - \Da_0),
    \end{align*}
    where we have defined $\Dh := e^{w_a} (\Bh(\I e^{-\rho_a}))_\psi$, $\Da := e^{\overline{w_{\overline{a}}}} (\Ba(\I e^{-\overline{\rho_{\overline{a}}}}))_\psi$, and where $\partial g_+ = \dbar g_- = 0$, so that $g_\pm$ are expressed as Cauchy formulas in terms of their boundary conditions
    \begin{align*}
	g_+|_{\partial M} = -i (\I - \Dh|_{\partial SM})_0, \qquad g_-|_{\partial M} = i (\I - \Da|_{\partial SM})_0.
    \end{align*}
\end{theorem}

\todo[inline]{The $g_\pm$ terms pretty much imply that their boundary traces satisfy some conditions w.r.t. the Hilbert transform of the domain. This is reminiscent of \cite[Lemma 4.1]{Kazantsev2007}}

\begin{remark} Another way to view the formula for $g_s$ is as the orthogonal projection of $- \frac{i}{2} (\Dh_0 - \Da_0)$ onto $H^1_0(M)$, for the inner product 
    \begin{align*}
	(u,v)_{\dot{H}^1(M)} := \int_M \nabla u \cdot \overline{\nabla v}. 
    \end{align*} 
\end{remark}

\begin{proof}[Proof of Theorem \ref{thm:inversionfh0}]
    Since $e^{-w_a}$ is a holomorphic solution of $X u - au =0$, we have
    \begin{align*}
	X (u e^{-w_a}) = - (g_0 + X_\perp g_s) e^{-w_a} = - b (x,v), 
    \end{align*}
    where $b$ is of the form $b_{-1} + b_0 + \O_{\ge 1}$ with, in particular, $b_{-1} = -\frac{1}{i} \eta_- g_s$. Thanks to Proposition \ref{prop:holomorphization}, defining $v:= ue^{-w_a}$, the function $\vh = v - (\Bh(v|_{\partial SM}))_\psi$ is holomorphic and satisfies $X \vh = -b$. Then defining $\uh := e^{w_a} \vh = u - e^{w_a} (\Bh( u e^{-w_a}|_{\partial SM}))_\psi = u - \Dh$, $\uh$ solves the equation 
    \begin{align}
	X \uh + a \uh = -g_0 -X_\perp g_s.    
	\label{eq:transuh}
    \end{align}
    Similarly using $e^{-\overline{w_{\overline{a}}}}$, an antiholomorphic solution of $X u - au =0$, the function $\ua = u - e^{\overline{w_{\overline{a}}}} (\Ba( ue^{-\overline{w_{\overline{a}}}}|_{\partial SM}))_\psi = u-\Da$ is antiholomorphic and solves
    \begin{align}
	X \ua + a \ua = -g_0-X_\perp g_s.
	\label{eq:transua}
    \end{align}
    Projecting \eqref{eq:transuh} onto $H_{-1}$ and \eqref{eq:transua} onto $H_1$, we obtain 
    \begin{align*}
	\eta_- \uh_0 = \frac{1}{i} \eta_- g_s &\qquad \Leftrightarrow \qquad e^{-i\theta}\ \dbar (g_s -i\uh_0) = 0, \\
	\eta_+ \ua_0 = -\frac{1}{i} \eta_+ g_s &\qquad \Leftrightarrow \qquad e^{i\theta}\ \partial (g_s + i\ua_0) = 0.
    \end{align*}
    This implies the relations: 
    \begin{align}
	g_s - i\uh_0 = g_+ \in \ker \partial, \qquad g_s + i\ua_0 = g_- \in \ker \dbar.
	\label{eq:relh0}    
    \end{align}
    which, since $g_s$ vanishes at the boundary, completely determines $g_\pm$ from their boundary values, which are in turn determined from the boundary values of $\uh$ and $\ua$ (via a Cauchy formula), which are known from data. Taking the half-sum, we obtain
    \begin{align*}
	g_s = \frac{1}{2} (g_+ + g_-) - \frac{i}{2} (\uh_0 - \ua_0) = \frac{1}{2} (g_+ + g_-) - \frac{i}{2} (\Dh_0 - \Da_0),
    \end{align*}
    where the right-hand side is completely determined by data. On to the determination of $g_0$, we project the equation $Xu + au = -g_0 - X_\perp g_s$ onto $H_0$ to make appear
    \begin{align*}
	g_0 = -\eta_+ u_{-1} - \eta_- u_1 - a u_0,
    \end{align*}
    and show how to determine each term from the data. Since $\uh$ is holomorphic, then $\uh_{-1} = 0 = u_{-1} - \Dh_{-1}$, so $u_{-1} = \Dh_{-1}$. Since $\ua$ is antiholomorphic, $\ua_1 = 0 = u_1 - \Da_1$, so $u_1 = \Da_1$. Finally, 
    \begin{align*}
	u_0 = \frac{1}{2} (\uh_0 + \Dh_0 + \ua_0 + \Da_0) \stackrel{\eqref{eq:relh0}}{=} \frac{1}{2} (\Dh_0 + \Da_0) + \frac{i}{2} (g_+ - g_-).  
    \end{align*}
    We arrive at the following formula for $g_0$ 
    \begin{align*}
	g_0 = -\eta_+ \Dh_{-1} - \eta_- \Da_1 - \frac{a}{2} \left( \Dh_0 + \Da_0 + i(g_+ - g_-) \right). 
    \end{align*}
    Theorem \ref{thm:inversionfh0} is proved.
\end{proof}

\todo[inline]{Combining all formulas, does the holomorphization operator admit a simple expression ? Or a easy-to-describe action on a given basis ?}

\appendix

\section{Proof of Proposition \ref{prop:Wk}}

\begin{proof}[Proof of Proposition \ref{prop:Wk}] A first calculation shows that 
    \begin{align}
	\begin{split}
	    \frac{u'_{k,k}}{\cos\alpha} = \frac{\sqrt{2}}{\pi} e^{ik\beta} \frac{e^{i(2k+1)\alpha} + (-1)^k e^{-i\alpha}}{e^{i\alpha} + e^{-i\alpha}} &= \frac{\sqrt{2}}{\pi} e^{ik\beta} (-1)^k \frac{1-(-e^{2i\alpha})^{k+1}}{1-(-e^{2i\alpha})} \\
	    &= \frac{\sqrt{2}}{\pi} e^{ik\beta} (-1)^k \sum_{p=0}^k (-1)^p e^{2ip\alpha},
	\end{split}
	\label{eq:firstcal}	
    \end{align}
    thus upon defining 
    \begin{align}
	J_{k,p} (\x) := \frac{1}{2\pi} \int_{\Sm^1} \left( e^{ik\beta} e^{2ip\alpha} \right)_\psi (\x,\theta)\ d\theta, \qquad 0\le p\le k,
	\label{eq:Jkp}
    \end{align}
    we have 
    \begin{align*}
	\left( \left( \frac{u'_{k,k}}{\cos\alpha} \right)_\psi \right)_0 = (-1)^k \frac{\sqrt{2}}{\pi} \sum_{p=0}^k (-1)^p J_{k,p}, \qquad 2 \left( \left( u_{k,k} \right)_\psi \right)_0 = (-1)^k \frac{\sqrt{2}}{\pi} \left( J_{k,0} + (-1)^k J_{k,k} \right). 
    \end{align*}
    The proof will then be complete once we prove that
    \begin{align}
	J_{k,0}(\x) = (-1)^k J_{k,k}(\x) = \frac{1}{2} (x+iy)^k, \quad\text{and} \quad J_{k,p}(\x) = 0, \quad p\notin \{0,k\}.
	\label{eq:WTSprop}
    \end{align}
    We rewrite 
    \begin{align*}
	J_{k,p} (\x) = \frac{1}{2\pi} \int_{\Sm^1} e^{ik\beta_-(\x,\theta)} e^{2ip\ \alpha_-(\x,\theta)}\ d\theta,
    \end{align*}
    where $(\beta_-(\x,\theta),\alpha_-(\x,\theta))\in \partial_+ SM$ denote the fan-beam coordinates of the unique geodesic passing through $(\x,\theta)\in SM$. We compute $(\beta_-, \alpha_-)$ explicitly after parameterizing $\x = \rho e^{i\beta}$ in polar coordinates. To this end, we will also need $\tau(\rho e^{i\beta},\theta)$, the first arrival time to the boundary of the geodesic passing through $(\rho e^{i\beta},\theta)$. We have, by direct computation of the unique positive root $t$ to the equation $|\rho e^{i\beta} + t e^{i\theta}|^2 = 1$, that 
    \begin{align*}
	\tau (\rho e^{i\beta}, \theta) = - \rho \cos(\theta-\beta) + \sqrt{1-\rho^2 \sin^2 (\theta-\beta)} = \tau (\rho,\theta-\beta). 
    \end{align*}
    Then, $\beta_-(\rho e^{i\beta},\theta)$ is such that 
    \begin{align*}
	e^{i\beta_- (\rho e^{i\beta}, \theta)} = \rho e^{i\beta} + \tau (\rho e^{i\beta}, \theta + \pi) e^{i(\theta+\pi)},  
    \end{align*}
    which yields the relation
    \begin{align*}
	e^{i\beta_-(\rho e^{i\beta}, \theta)} = - e^{i\theta} \left( \sqrt{1-\rho^2 \sin^2 (\theta-\beta)} + i \rho \sin (\theta-\beta) \right).
    \end{align*}
    On to the calculation of $\alpha_- (\x,\theta)$, recalling that $\theta = \beta_- + \alpha_- + \pi$ and using the previous relation, we arrive at
    \begin{align*}
	e^{i\alpha_-( \rho e^{i\beta}, \theta)} = - e^{i (\theta-\beta_-( \rho e^{i\beta}, \theta))} = \left( \sqrt{1-\rho^2 \sin^2 (\theta-\beta)} + i \rho \sin (\theta-\beta) \right)^{-1}. 
    \end{align*}
    On to the computation of $J_{k,p}$ where we first restrict to $k-2p\ge 0$:
    \begin{align*}
	2\pi J_{k,p} (\rho e^{i\beta}) &= \int_{\Sm^1} e^{ik \beta_-(\rho e^{i\beta}, \theta)} e^{2ip\ \alpha_-(\x,\theta)}\ d\theta \\
	&= (-1)^k \int_{\Sm^1} e^{ik\theta} \left( \sqrt{1-\rho^2 \sin^2 (\theta-\beta)} + i \rho \sin (\theta-\beta) \right)^{k-2p}\ d\theta \\
	&= (-1)^k e^{ik\beta} \int_{\Sm^1} e^{ik\theta} \left( \sqrt{1-\rho^2 \sin^2 \theta} + i \rho \sin \theta \right)^{k-2p}\ d\theta \\
	&= (-1)^k e^{ik\beta} \sum_{q=0}^{k-2p} \binom{k-2p}{q} \int_{\Sm^1} e^{ik\theta} (1-\rho^2 \sin^2 \theta)^{\frac{q}{2}} (i \rho \sin \theta)^{k-2p-q}\ d\theta. 
    \end{align*}
    In the sum above, the integrands are odd w.r.t. $\theta\mapsto \theta+\pi$ whenever $q$ is odd, and therefore these terms vanish. In addition, when $q$ is even, the harmonic content of $(1-\rho^2 \sin^2 \theta)^{\frac{q}{2}} (i \rho \sin \theta)^{k-2p-q}$ lies in the span of $\{ e^{im\theta},\ -k+2p \le m \le k-2p\}$. In particular, if $p\ne 0$ (with $k-2p>0$), the inner product of such an integrand with $e^{ik\theta}$ is zero, and thus $J_{k,p} = 0$. Now, when $p=0$, we compute: 
    \begin{align*}
	J_{k,0} (\rho e^{i\beta}) &= (-1)^k e^{ik\beta} \frac{1}{2\pi} \int_{\Sm^1} e^{ik\theta} \left( \sqrt{1-\rho^2 \sin^2 \theta} + i \rho \sin \theta \right)^{k}\ d\theta \\
	&= (-1)^k e^{ik\beta}  c_{-k} \left( \left(\sqrt{1-\rho^2\sin^2\theta} + i\rho\sin\theta\right)^k \right) \\
	&= (-1)^k e^{ik\beta}  \sum_{q=0}^k \binom{k}{q} c_{-k} \left( (1-\rho^2\sin^2\theta)^\frac{q}{2} (i\rho\sin\theta)^{k-q} \right),
    \end{align*} 
    where by $c_{-k}(g)$ we denote the coefficient in front of $e^{-ik\theta}$ in the Fourier series expansion of $g$. Now, for $q$ odd, the function $\theta\mapsto (1-\rho^2\sin^2\theta)^\frac{q}{2} (i\rho\sin\theta)^{k-q}$ has the opposite parity as $k$ w.r.t. the involution $\theta\mapsto \theta+\pi$, so the $c_{-k}$ coefficient vanishes in the case of $q$ odd. If $q$ is even, say $q=2s$, we rewrite
    \begin{align*}
	( 1-\rho^2\sin^2\theta)^\frac{q}{2} (i\rho\sin\theta)^{k-q} = (1-\rho^2\sin^2\theta)^s (i\rho\sin\theta)^{k-2s}.
    \end{align*}
    Expanding the first factor in the right hand side using the binomial formula, we obtain a series in powers of $\sin\theta$, where the only term with a nonzero $c_{-k}$ coefficient is 
    \begin{align*}
	(-\rho^2\sin^2\theta)^s (i\rho\sin\theta)^{k-2s} = (i\rho\sin\theta)^k = \rho^k (-1)^k 2^{-k} e^{-ik\theta} + \sum_{\ell\ne k} c_{\ell} e^{i\ell \theta}.  
    \end{align*}
    We can then write $c_{-k} \left( (1-\rho^2\sin^2\theta)^\frac{q}{2} (i\rho\sin\theta)^{k-q} \right) = (-1)^k \rho^k 2^{-k}\ \frac{1+(-1)^q}{2}$. Back to the calculation of $J_{k,0}$, we then obtain 
    \begin{align*}
	J_{k,0} (\rho e^{i\beta}) &= (-1)^k e^{ik\beta}  \sum_{q=0}^k \binom{k}{q} (-1)^k \rho^k 2^{-k}\ \frac{1+(-1)^q}{2}  \\
	&=  \frac{1}{2} (\rho e^{i\beta})^k 2^{-k} \left( (1+1)^k + (1-1)^k \right) \\
	&= \frac{1}{2} (\rho e^{i\beta})^k.
    \end{align*}
    In order to treat the missing values $k-2p<0$, changing variables $\theta\mapsto \theta+\pi$ in the definition of $J_{k,p}$ and using the fact that 
    \begin{align*}
	(\beta_-(\x,\theta+\pi), \alpha_- (\x,\theta+\pi)) = \SS_A (\beta_- (\x,\theta), \alpha_-(\x,\theta)) = (\beta_-(\x,\theta) + \pi + 2\alpha_- (\x,\theta), -\alpha_- (\x,\theta)),
    \end{align*}
    we arrive at the relation 
    \begin{align*}
	(-1)^p J_{k,p} = (-1)^{k-p} J_{k,k-p}.
    \end{align*}
    In particular, this implies that $(-1)^k J_{k,k} (\rho e^{i\beta}) = J_{k,0} (\rho e^{i\beta}) = \pi (\rho e^{i\beta})^k$, and that for all other values $0<p<k$, $J_{k,p} = 0$. Proposition \ref{prop:Wk} is proved.    
\end{proof}

\bibliographystyle{siam}

\end{document}